\documentclass{amsart}

\usepackage{color, mathtools}
\usepackage{verbatim}

\mathtoolsset{showonlyrefs}

\newcommand{\ddb}{\sqrt{-1}\partial\overline{\partial}}

\renewcommand{\[}{\begin{equation} \begin{aligned} }
\renewcommand{\]}{\end{aligned} \end{equation}}

\newtheorem{thm}{Theorem}
\newtheorem{prop}[thm]{Proposition}
\newtheorem{lem}[thm]{Lemma}
\newtheorem{cor}[thm]{Corollary}

\theoremstyle{definition}

\numberwithin{equation}{section}

\begin{document}
\title{Degenerations of $\mathbf{C}^n$ and Calabi-Yau metrics}

\author{G\'{a}bor Sz\'{e}kelyhidi}
\address{Department of Mathematics\\
         University of Notre Dame\\
         Notre Dame, IN 46556}
\email{gszekely@nd.edu}

\begin{abstract}
  We construct infinitely many complete Calabi-Yau metrics on
  $\mathbf{C}^n$ for $n \geq 3$, with
  maximal volume growth, and singular tangent cones at infinity. In
  addition we construct Calabi-Yau metrics in neighborhoods of certain
  isolated singularities whose tangent cones have singular cross
  section, generalizing work of Hein-Naber~\cite{HNpreprint}. 
\end{abstract}

\maketitle
\section{Introduction}
Since the seminal work of Yau~\cite{Yau78}, Calabi-Yau metrics have
been studied extensively in K\"ahler geometry. Beyond the case of
compact K\"ahler manifolds, there have been many constructions of
non-compact Calabi-Yau metrics with various behaviors at infinity,
by Cheng-Yau~\cite{CY80}, Tian-Yau~\cite{TY91, TY90} and
others. In this paper we are concerned with constructing new
non-compact Calabi-Yau manifolds that have Euclidean volume growth at
infinity. In this case there exists a tangent cone at infinity~\cite{CC97}, which
is expected to be unique~\cite{CM14, DS15}, and a natural
problem is to try constructing complete Calabi-Yau metrics with
prescribed tangent cones. 
There are many such constructions in the literature, with tangent
cones that have smooth links~\cite{vC10,Goto12,CH13, CH15}, as well as singular
links~\cite{BG97, Joy00, CDR16}. 

Our work pushes these methods further, obtaining a
large class of new examples on $\mathbf{C}^n$, for $n\geq 3$. 
In particular these give counterexamples to a conjecture of Tian~\cite[Remark
5.3]{TianAspects}, stating that the flat metric is the unique
Calabi-Yau metric on $\mathbf{C}^n$ with maximal volume growth. 
Some of these examples have very
recently been independently obtained by Li~\cite{Li17} and
Conlon-Rochon~\cite{CR17}, using somewhat different techniques. 
See Section~\ref{sec:comparison} for a
comparison with our work. In addition we also
construct Calabi-Yau metrics in neighborhoods of certain isolated
singularities, with singular tangent cones, extending unpublished
work of Hein-Naber~\cite{HNpreprint}. 

Consider the hypersurface $X_1 \subset \mathbf{C}^{n+1}$ given by the
equation
\[ z + f(x_1,\ldots, x_n) = 0, \]
for a polynomial $f$, so that $X_1$ is biholomorphic to
$\mathbf{C}^n$. Suppose in addition that if we let the $x_i$ have
weights $w_i > 0$, then $f$ has degree $d > 1$. If we write
$F_t(z, x_1,\ldots, x_n) = (tz, t^{w_1}x_1, \ldots, t^{w_n}x_n)$, then $F_t^{-1}
X_1$ has the equation
\[ t^{1-d} z + f(x_1,\ldots, x_n) = 0, \]
and so if $d > 1$, then $F_t^{-1}X_1 \to X_0$ as $t\to\infty$, where 
\[ X_0 = \mathbf{C} \times f^{-1}(0). \]
Suppose that $X_0$ admits a (singular) Ricci flat cone metric
$\omega_0$ whose homothetic transformations are the maps $F_t$. It is then natural to
expect that we can use $\omega_0$ to define an asymptotically Ricci
flat metric on $X_1$ whose tangent cone at infinity is $(X_0,
\omega_0)$, that we can perturb to a complete Calabi-Yau metric on
$X_1$ with the same tangent cone. This almost fits into the class of
problems studied by Conlon-Hein~\cite{CH13}, except for the fact that
$X_0$ has more than just an isolated singularity at the origin. 

We restrict ourselves to the simplest situation, when $V_0 = f^{-1}(0) \subset
\mathbf{C}^n$ has an isolated normal singularity at the origin, so that $X_0$
is only singular along $\mathbf{C}\times \{0\}$. In particular $n \geq
3$. If we focus
on the slice $\{1\} \times V_0 \subset X_0$ near the singular ray, 
then the corresponding slice of $F_t^{-1}X_1$ is a smoothing
\[ t^{1-d} + f(x_1,\ldots, x_n) = 0 \]
of the cone $V_0$. Our strategy is to write down a metric on $X_1$ by
combining a perturbation of the singular Calabi-Yau metric on $X_0$
with a Calabi-Yau metric on this smoothing of $V_0$ near the
singular rays. The Calabi-Yau metric on the smoothing is obtained
using the results of Conlon-Hein~\cite{CH13}. 
Our main result then is as follows.

\begin{thm}\label{thm:1}
  Under the assumption that the hypersurface
 $V_0$ admits a Calabi-Yau cone metric,
  there is a complete Calabi-Yau metric on $\mathbf{C}^n$ whose
  tangent cone at infinity is $\mathbf{C}\times V_0$. 
\end{thm}

As an example we can consider the case when $n=3$, and 
\[ f(x_1,x_2,x_3) = x_1^2 + x_2^2 + x_3^k \]
for $k \geq 2$. The hypersurfaces $f^{-1}(0)$ are the $A_{k-1}$
singularities and they all admit flat cone metrics, being cyclic
quotients of $\mathbf{C}^2$. Therefore we obtain
infinitely many complete Calabi-Yau metrics on $\mathbf{C}^3$ with
tangent cones $\mathbf{C}\times A_{k-1}$ at infinity. Simple higher
dimensional examples can be obtained by 
taking products with $\mathbf{C}$. The
corresponding metric when $k=2$ has recently been constructed by
Li~\cite{Li17} and Conlon-Rochon~\cite{CR17} independently. 
Another example in higher dimensions is obtained with 
\[ f(x_1,\ldots, x_n) = x_1^2 + \ldots + x_n^2, \]
for $n \geq 3$, i.e. the $A_1$ singularity, which admits the Stenzel
cone metric. We therefore obtain complete Calabi-Yau metrics
on $\mathbf{C}^n$ with tangent cone $\mathbf{C} \times A_1$ at
infinity. 

Consider now the hypersurface $X_1 \subset
\mathbf{C}^{n+1}$ with an isolated singularity $0\in X_1$, given by
\[ z^p + f(x_1,\ldots, x_n) = 0, \]
where $p > 1$, and $f$ is as above. With $F_t$ as before, the equation
of $F_t^{-1}X_1$ is now 
\[ t^{p-d}z^p + f(x_1,\ldots, x_n) = 0, \]
and so if $p > d$, then we have $F_t^{-1}X_1 \to X_0$ as $t\to 0$, with
$X_0 = \mathbf{C}\times V_0$. It is therefore natural to expect that
$X_1$ admits a metric which is asymptotically Calabi-Yau as we
approach the singular point, and whose tangent cone at 0 is
$\mathbf{C}\times V_0$. Using an argument that is essentially
identical to part of the proof of Theorem~\ref{thm:1}, we prove that this is the case, and
in fact this metric can be perturbed to be Calabi-Yau in a
neighborhood of the singular point. 
This result generalizes unpublished work of Hein-Naber~\cite{HNpreprint},
which applies to the case when $f = x_1^2 + \ldots + x_n^2$, and $p >
2\frac{n-1}{n-2}$, using different techniques. 
\begin{thm}\label{thm:2}
  If $p > d$, then $X_1$ admits a Calabi-Yau metric on a neighborhood
  of the singular point $0$, whose tangent cone at $0$ is
  $\mathbf{C}\times V_0$.  
\end{thm}

It is likely that such constructions can be applied in a much more general
setting, and we focus on these relatively explicit examples for
simplicity. In general suppose that $X_0 \subset \mathbf{C}^N$ is
a subvariety which admits a possibly singular Calabi-Yau cone metric
whose homothetic (Reeb) vector field is $\xi = \sum w_i
z_i \partial_{z_i}$ for weights $w_i > 0$. 
This vector field generates a real one-parameter
group of biholomorphisms $F_t$ of $\mathbf{C}^N$. Let $X_1 \subset
\mathbf{C}^N$ be another subvariety, and suppose that one of the
following two conditions holds:
\begin{enumerate}
  \item $\lim_{t\to\infty} F_t^{-1} X_1 = X_0$, 
  \item or $\lim_{t\to 0} F_t^{-1} X_1 = X_0$. 
\end{enumerate} 
Case (1) is the setting of Theorem~\ref{thm:1}, and here one expects
to be able to construct a Calabi-Yau metric on $X_1$ near
infinity, with tangent cone $X_0$ at infinity. 
If $X_1$ is smooth, then one can hope to deform this into a
global Calabi-Yau metric on $X_1$ using techniques of
Tian-Yau~\cite{TY91}. 
In case (2), by contrast, we necessarily have $0\in X_1$, and 
one expects to be able to contruct a 
Calabi-Yau metric on $X_1$ on a neighborhood of $0$, whose tangent
cone at $0$ is $X_0$, as in Theorem~\ref{thm:2}. It seems likely that
our methods can be generalized to prove these expectations whenever
the singular behavior of $X_0$ has enough structure, such as the
iterated edge spaces of Degeratu-Mazzeo~\cite{DM14}. 

Note that for a given $X_1$, one typically expects infinitely many
possible $X_0$ to fit in case (1) above, as in Theorem~\ref{thm:1},
but there should be at most one $X_0$ fitting into case (2). This is
because the tangent cone of a Calabi-Yau metric on $X_1$ at the
singularity is expected to be independent of the metric (see
Hein-Sun~\cite{HS16} for a special case of this).

This discussion fits into the
framework developed by Donaldson-Sun~\cite{DS15} for analyzing the
metric tangent cones of singular Calabi-Yau metrics under certain
assumptions. They show that (under additional assumptions) the metric
tangent cone $C(Y)$ at 0 of a Calabi-Yau metric on $X_1$ is an affine
algebraic variety which can be obtained from $X_1$ by a ``two step''
degeneration. In a first step we associate to $X_1$ its weighted
tangent cone $W$ at $0$ for a suitable canonical valuation at $0$. This
valuation is obtained, roughly speaking, 
from the rate of growth of germs of holomorphic functions on $X_1$
with respect to the Calabi-Yau metric. The metric tangent cone
$C(Y)$ is then obtained by a further degeneration (or test-configuration)
of $W$. One can think of these
two steps as being analogous to the Harder-Narasimhan and
Jordan-H\"older filtrations of an unstable, respectively semistable,
vector bundle. Note that in the setting of Theorem~\ref{thm:2} both $W$ and
$C(Y)$ equal $X_0$ and so it would be interesting to construct examples
where $X_1 \ne W\ne C(Y)$. 

\subsection{Outline}
We now give an outline of the proof of
Theorem~\ref{thm:1}. First, in Section~\ref{sec:smoothing1} we
 construct a complete Calabi-Yau metric
on the smoothing $V_1 = f^{-1}(1)$ of the Calabi-Yau cone $V_0$. This
is already contained in work of Conlon-Hein~\cite{CH13}, but we give
some of the main points since they are a simpler version of what we
need later. 

In Section~\ref{sec:approx1} we write down a metric $\omega$ on $X_1$, which is
approximately Calabi-Yau near infinity, modeled on the product metric
on $X_0 = \mathbf{C}\times V_0$. This product is singular along the
rays $\mathbf{C}\times \{0\}$, and so near these singular
rays we glue in suitable scaled copies of the Calabi-Yau metric on
$V_1$ using cutoff functions. The key technical result is the estimate in
Proposition~\ref{prop:decayest2} of the Ricci potential of $\omega$,
in suitable weighted H\"older spaces. 

The next step, in Section~\ref{sec:perturb}, is to modify the metric $\omega$
to improve the decay of its Ricci potential. This is analogous to
Lemma 2.12 in Conlon-Hein~\cite{CH13} and is based on inverting the
Laplacian in suitable weighted spaces. After this we can further
deform $\omega$ to a Calabi-Yau metric using a non-compact version of
Yau's Theorem developed by Tian-Yau~\cite{TY91}. We use the
version of this due to Hein~\cite{HeinThesis}. 

The technical heart of the paper is in
Sections~\ref{sec:modelspace} and \ref{sec:Laplaceinverse}, inverting the
Laplacian in suitable weighted spaces.  On asymptotically conical
manifolds with smooth link at infinity there is a well developed
theory for this, going back to Lockhart-McOwen~\cite{LM85}, however the
tangent cone $X_0$ of $(X_1, \omega)$ has a singular cross section -
it is the double suspension of the link of $V_0$, which has a circle
of singularities modeled on $V_0$. There have been several works
dealing with the Laplacian on similar spaces, most notably the theory
of QALE spaces due to Joyce~\cite{Joy00}, and the more general QAC spaces studied
by Degeratu-Mazzeo~\cite{DM14}. On the one hand $(X_1,\omega)$ does
not quite fit into the QAC framework, but more crucially, when solving
the equation $\Delta u = f$, these works require at least quadratic
decay of $f$ (at least away from the singular rays), since they essentially rely
on taking a convolution with the Green's function. In our application,
however, we also deal with $f$ which have slower decay (this is the
case when the degree $d \leq 3$). 

In order to solve $\Delta u =f$ for $f$ which do not decay
sufficiently fast, instead of analysing the Green's function, we
construct an approximate inverse for $\Delta$ on suitable local
patches, which we can then glue together using cutoff functions. This
is analogous to the method employed in many other geometric gluing
problems (see e.g. Donaldson-Kronheimer~\cite[Chapter 7]{DK90} or
\cite{GSz10}).  One of the local models that we have to analyze is
the space $X_0\setminus(\mathbf{C}\times \{0\})$ which the part of
$X_1$ away from the singular rays of $X_0$ 
is modeled on.  The other model space is $\mathbf{C}\times
V_1$, on which neighborhoods of those points in $X_1$ are modeled on 
that are close to the singular rays of $X_0$. The basic strategy for
solving $\Delta u = f$ approximately is to first decompose $f$ into
pieces that are supported inside the model neighborhoods, then invert
the Laplacian on the model spaces, and finally patch the results back
together using cutoff functions. We then need to control the errors
that are introduced by the cutoff functions on the one hand, and by
replacing the metric $\omega$ on the model neighborhoods by the
corrsponding model metrics. 

\subsection{Relation to other recent works}\label{sec:comparison}
As this project was nearing completion, 
 two other works appeared that have significant overlap with our
 results. One is the paper of Yang Li~\cite{Li17} mentioned above. It deals
 with the particular case of constructing a metric on $\mathbf{C}^3$
 with tangent cone $\mathbf{C}\times A_1$ at infinity. It relies on
 fairly explicit calculations, exploiting the symmetries of the
 situation. It is likely that his approach gives better asymptotic
 understanding of this metric, and this may be important in
 applications to gluing problems. 

The other, even more recent, paper is Conlon-Rochon~\cite{CR17},
constructing Calabi-Yau metrics on $\mathbf{C}^n$ with tangent cones
$\mathbf{C}\times V_0$ for suitable Calabi-Yau cones $V_0$. 
It relies on extending the work of Degeratu-Mazzeo~\cite{DM14} to a
class of ``warped QAC'' metrics, and as such it still requires faster
than quadratic decay of the Ricci potential of the approximate
solution. In terms of our Theorem~\ref{thm:1} this is the case when
the degree $d > 3$. One important case that this does not cover is
when $f(x) = x_1^2 + \ldots + x_n^2$, and $n > 3$, since here
$d = 2\frac{n-1}{n-2} \leq 3$, although in a subsequent version of
\cite{DM14} the authors have overcome this issue by constructing a
metric whose Ricci potential decays more rapidly. 
Nevertheless, our approach to inverting the Laplacian
also has the advantage that the same method applies to
Theorem~\ref{thm:2}. 

\subsection*{Acknowledgements}
I am grateful to Hajo Hein for multiple helpful discussions. This work
was supported in part by NSF grant DMS-1350696. In addition I thank
the anonymous referees for their careful reading of the paper and helpful
suggestions. 

\section{The smoothing model}\label{sec:smoothing1}
Suppose that $V_0\subset \mathbf{C}^n$ is a hypersurface given by the
equation $f(x_1,\ldots, x_n)=0$, with an isolated singularity at the
origin. Suppose that we have a weight vector $\xi = (w_1,\ldots, w_n)$
with positive, possibly non-integral entries, giving rise to the action
\[ t\cdot (x_1,\ldots, x_n) = (t^{w_1}x_1,\ldots t^{w_n}x_n) \]
on $\mathbf{C}^n$ for $t > 0$. 
We denote the degree of $f$ under this action by $d$,
i.e. $f(t\cdot x) = t^d f(x)$. This action generates the action of a
complex torus $T^c$ on $\mathbf{C}^n$, which fixes $V_0$. We write
$T\subset T^c$ for the maximal compact torus. We will also write
$z\cdot x$ for $z\in \mathbf{C}^*$, which means that we choose a
branch of $\log z$ to define the non-integer powers of $z$. The choice
of branch will not matter. 

There is a nowhere vanishing holomorphic
$(n-1)$-form $\Omega$ on $V_0\setminus \{0\}$ given by 
\[ \Omega = \frac{ dx_2 \wedge dx_3 \wedge\ldots \wedge
  dx_n}{ \partial_{x_1} f} \]
where $\partial_{x_1} f\not=0$, and by similar expressions
where $\partial_{x_i} f \not=0$ for $i > 1$. We assume that $\Omega$
has degree $n-1$ under the action above, which is equivalent to the identity 
\[ \sum_{i=1}^n w_i = d + n -1. \]

Finally we suppose that $V_0$ admits a Ricci flat K\"ahler cone metric
$\omega_{V_0}$, whose homothetic transformations are given by the
action above. Equivalently we have $\omega_{V_0}^{n-1} =
(\sqrt{-1})^{(n-1)^2}\Omega \wedge \overline{\Omega}$.  This setup
has been studied extensively in the literature, in particular in
relation to Sasakian geometry (see \cite{GMSY}).

\begin{lem}\label{lem:d>2}
  Under these assumptions, unless $V_0\cong \mathbf{C}^{n-1}$, we have
  $d > 2$. 
\end{lem}
\begin{proof}
  Let us denote by $w_{min}$ the smallest weight. The Lichnerowicz
  obstruction of Gauntlett-Martelli-Sparks-Yau~\cite{GMSY} implies
  that since $V_0$ admits a Calabi-Yau cone metric, we have $w_{min} >
  1$ (if $w_{min}=1$ then necessarily $V_0\cong \mathbf{C}^{n-1}$). It
  then follows that $d > 2$, since $d \geq 2w_{min}$, unless $f$ is
  linear. 
\end{proof}

Because of this Lemma we can assume throughout that $d > 2$. Otherwise
$V_0 \cong \mathbf{C}^{n-1}$ in which case Theorems~\ref{thm:1}
and \ref{thm:2} are clear. 
Here we are interested
in the existence of complete Ricci flat metrics on the smoothing
$V_1$ of $V_0$, given by the equation $1 + f(x) = 0$. This question
was addressed by Conlon-Hein~\cite{CH13}, and here we review the main
points since our result can be thought of as a generalization of
their work. 

The Ricci flat metric on $V_0$ is given by $\omega_{V_0} = \ddb r^2$,
where $r$ is the 
radial distance from the vertex of the cone. The idea is to use
$r^2$ to write down a K\"ahler potential on $V_1$ near infinity, which is
approximately Ricci flat. Since $V_1$ is asymptotic to $V_0$ near
infinity, a natural approach, followed by Conlon-Hein, is to simply use
an orthogonal projection from $V_1$ to $V_0$ (outside a bounded set)
to pull back the potential $r^2$. We follow a slight variant of this
method. First let us define the function $R:\mathbf{C}^n\to
\mathbf{R}$, by letting $R=1$ on the Euclidean unit sphere, and
extending $R$ to have degree 1 under the action of $F_t$. As shown in
He-Sun~\cite [Lemma 2.2]{HS12}, then 
the form $\ddb R^2$ defines a cone metric on $\mathbf{C}^n$. By the
homogeneity, its 
restriction to $V_0$ is uniformly equivalent to
$\omega_{V_0}$. In particular the function $R$ restricted to $V_0$ is
also uniformly equivalent to the distance function $r$. 
We now take
any smooth extension of $r$ from $V_0$ to
$\mathbf{C}^n\setminus\{0\}$, which has degree 1 under the action
$\xi$. This can be defined by first extending $r$ on the sphere
$R=1$, and then extending it further by homogeneity. We will continue
to write $r$ for this extended function. 

We claim the following. 
\begin{prop}\label{prop:hest1}
  There exists a constant $A > 0$ such that
  \begin{itemize}
    \item[(a)]   The form $\omega = \ddb r^2$ is positive definite on $V_1
      \cap \{R > A\}$. 
    \item[(b)]  The Ricci potential 
      \[ h = \log \frac{(\ddb r^2)^{n-1}}{(\sqrt{-1})^{(n-1)^2}\Omega \wedge
        \overline{\Omega}} \]
      of $(V_1, \omega)$ satisfies
      \[  \label{eq:hest}\nabla^i h = O(R^{-d-i}), \]
      as $R\to\infty$, for all $i \geq 0$, measured with respect to
      $\omega$.
      \end{itemize}
\end{prop}
\begin{proof}
  Given large $K$, let us analyze the region $\{K/2 < R <2K\}$ by
  scaling the metric by a factor of $K^{-1}$. At the same time let us
  introduce rescaled coordinates $\tilde{x} = K^{-1}\cdot x$, and
  $\tilde{r} = K^{-1} r, \tilde{R} = K^{-1}R$. The rescaled metric is given by
  \[ K^{-2} \omega = \ddb \tilde{r}^2, \]
  and in the rescaled coordinates the equation of $V_1$ is
  \[ K^{-d} + f(\tilde{x}) = 0. \]
   Let us write this as $V_{K^{-d}}$. 
   Restricted to the annular region $\{1/2 < \tilde R < 2\}$, the
   submanifolds $V_{K^{-d}}$ converge in $C^\infty$ to $V_0$
   uniformly. In addition $\tilde{r}^2$ as a function of $\tilde{x}$
   is independent of $K$ by the homogeneity of $r$. 
   
   Using the implicit function theorem we can
   cover $V_{K^{-d}} \cap \{1/2 < \tilde{R} < 2\}$ by a finite number of
   coordinate balls $B_i$ in each of which $V_{K^{-d}}$ is a hyperplane, and moreover
   we can find holomorphic functions $G_i : B_i \to \mathbf{C}^n$
   such that $G_i(V_{K^{-d}})\subset V_0$, and  $G_i(y) = y +
   O(K^{-d})$. Since we also control
   the derivatives of the $G_i$ and $\tilde r^2$ is a fixed smooth function, we obtain
   \[ \tilde r^2 - G_i^*\tilde r^2 = O(K^{-d}) \]
   in $C^\infty$,  on each of our coordinate balls. We measure
   derivatives here with respect to a fixed metric $\ddb \tilde R^2$.
   In particular once $K$ is sufficiently large, $\ddb \tilde r^2$ defines a
   positive definite form on $V_{K^{-d}}$, and in fact it is uniformly
   equivalent to $\ddb \tilde R^2$.  
  
   To control the Ricci potential, we simply need to compare $\Omega$ with
   $G_i^*\Omega$, since by assumption we have $\omega_{V_0}^n =
   (\sqrt{-1})^{(n-1)^2} \Omega \wedge\overline{\Omega}$. By the same reasoning
   we have $\Omega - G_i^*\Omega = O(K^{-d})$. Since the Ricci
   potential is invariant under scaling we find that on the annulus
   $\{K/2 < R < 2K\}$ we have
   \[ |\nabla^i h|_{K^{-2} \ddb R^2} = O(K^{-d}), \]
   where we indicate that we measure derivatives with respect to
   $K^{-2}\ddb R^2$. 
   Since $\ddb R^2$ is uniformly equivalent to $\ddb r^2$ (once $K$ is
   sufficiently large), this implies the result.  
\end{proof}

We can now write down a metric $\omega_{V_1}$ 
on $V_1$, which agrees with $\ddb r^2$ on the set where $R$ is
sufficiently large. One
way to do this is to consider the K\"ahler potential $C'
(1+|x|^2)^\alpha$ on $\mathbf{C}^n$ for
large $C'$. For sufficiently small $\alpha > 0$ this grows slower than
$r^2$, but if $C'$ is sufficiently large, then we can ensure that 
\[ C'(1 + |x|^2)^\alpha > r^2 \]
on the set where $R < 3K/2$, say. 
We can now define a regularized maximum (see
Demailly~\cite[\textsection 5.E]{Demailly})
\[ \Phi = \widetilde{\max} \Big\{ C'(1 + |x|^2)^\alpha,  r^2 \Big\}, \]
and let $\omega_{V_1} = \ddb \Phi |_{V_1}$. This defines a smooth
metric on $V_1$, which equals $\ddb r^2$ where $R$ is sufficiently
large. 

The estimate \eqref{eq:hest} says that the Ricci potential $h$ of
$\omega_{V_1}$ satisfies $h \in C^\infty_{-d}$ in terms of the
weighted spaces used by Conlon-Hein~\cite{CH13}. Their Theorem 2.1 then
implies that we can perturb $\omega_{V_1}$ to a Calabi-Yau
metric $\eta_{V_1} = \ddb \phi$ on $V_1$, where the decay of $\phi -
r^2$ can be controlled. Since by Lemma~\ref{lem:d>2} we have $d > 2$, 
from \cite[Theorem 2.1]{CH13} it follows that we can ensure $\phi -
r^2 \in C^\infty_{-c}(V_1)$ for some $c > 0$. If $d > 3$ then we
can even choose $c > 1$. We can write this estimate in the following
form: there are constants $C_i$, such that if $x$ lies in the region
where $1/2 < R < 2$, and $\lambda\cdot x \in V_1$ for some $\lambda >
1$, then 
\[ \label{eq:phiestimate}
  \left|\nabla^i  \Big[r^2 - \lambda^{-2} \phi (\lambda\cdot x)\Big]
  \right|_{\ddb R^2} < C_i \lambda^{-2-c}.\]
Moreover by the
uniqueness statment of \cite[Theorem 2.1]{CH13}, $\phi$ is invariant
under the action of $T$. 

\section{The approximate solutions on
  $\mathbf{C}^n$}\label{sec:approx1}
We now consider $X_1\subset
\mathbf{C}^{n+1}$ given by
\[ z + f(x_1,\ldots, x_n) = 0, \]
where $f$ is the polynomial from Section~\ref{sec:smoothing1}. Recall
that we have a Ricci flat cone metric on $V_0 = f^{-1}(0)\subset
\mathbf{C}^n$, whose distance function is $r$, and we have smoothly
extended $r$ so that it is defined on $\mathbf{C}^n\setminus \{0\}$
and has degree $1$ for the action $\xi$. 

The hypersurface $X_0 = \mathbf{C}\times V_0\subset
 \mathbf{C}^{n+1}$ then has a Ricci flat cone metric $\ddb(|z|^2 + r^2)$.
The corresponding homothetic scalings are given by the action with
weights $(1, w_1,\ldots, w_n)$. This metric is uniformly equivalent to
$\ddb \rho^2$, where
\[ \rho^2 = |z|^2 + R^2, \]
in terms of $R$ from above. 

We would like to define a metric on $X_1$ using the potential $|z|^2 +
r^2$ just as we did for $V_1$ above, but now $X_0$ and $|z|^2+r^2$ are singular along
$\mathbf{C}\times \{0\}$. Our approach is to use $|z|^2 + r^2$ as the
potential away from the singular rays, and a suitable scaling of the
potential $\phi$ on $V_1$ near the singular rays. 

Let us denote by $\gamma_1(s)$ a cutoff function satisfying
\[ \gamma_1(s) = \begin{cases} 1\, &\text{ if } s > 2 \\
  0\, &\text{ if } s < 1, \end{cases}\]
and write $\gamma_2 = 1 - \gamma_1$. We then define the approximate
solution on $X_1$, at least on the set where $\rho > P$ for
sufficiently large $P$, by
\[  \label{eq:approx1} \omega = \ddb\Big( |z|^2 + \gamma_1(R\rho^{-\alpha}) r^2 +
\gamma_2(R\rho^{-\alpha}) |z|^{2/d} \phi(z^{-1/d}\cdot x)\Big), \]
where $\alpha \in (1/d, 1)$ is to be chosen. 
As before we must choose a branch of $\log z$ to define $z^{-1/d}\cdot
x$, however the value of $\phi(z^{-1/d}\cdot x)$ is independent of
this choice since $\phi$ is $T$-invariant. Note moreover that if
\[ z + f(x) = 0, \]
then 
\[ 1 + f(z^{-1/d}\cdot x) = 0, \]
and so $z^{-1/d}\cdot x \in V_1$, where $\phi$ is defined. Writing
$\phi(x) = r^2 + \phi_{-c}(x)$ we have
\[ \omega =  \ddb\Big( |z|^2 + r^2 +
\gamma_2(R\rho^{-\alpha}) |z|^{2/d} \phi_{-c}(z^{-1/d}\cdot
x)\Big), \]
and the estimate \eqref{eq:phiestimate} implies that the term
involving $\phi_{-c}$ is of lower order than $|z|^2 + r^2$. It follows from this
that $\omega = \ddb\Phi$
where the potential $\Phi$ has the same growth rate as $|z|^2 +
r^2 = \rho^2$. In particular if $\omega$ is positive definite on
the set where $\rho > P$, we can argue as in the construction
of $\omega_{V_1}$ above, to construct a metric
on $X_1$ that agrees with $\omega$ where $\rho > 2P$. 

The holomorphic $n$-form
\[ \Omega = \frac{dz\wedge dx_2\wedge \ldots \wedge
    dx_n}{\partial_{x_1}f} \]
restricts to a nowhere vanishing $n$-form on $X_0$ as well as on
$X_1$, and as in the previous section we wish to estimate the Ricci
potential of $\omega$ with respect to $\Omega$. We have the
following in analogy with Proposition~\ref{prop:hest1}. 
\begin{prop}\label{prop:decayest2}
  Fix $\alpha\in (1/d,1)$. 
  The form $\omega$ defines a metric on the subset of $X_1$ where $\rho > P$, for
  sufficiently large $P$. For suitable constants $\kappa, C_i > 0$ and
  weight $\delta < 2/d$, 
  the Ricci potential $h$ of $\omega$
  satisfies, for large $\rho$, 
  \[ |\nabla^i h|_\omega < \begin{cases} 
      C_i \rho^{\delta-2-i}\, \text{ if } R > \kappa \rho \\
      C_i \rho^\delta R^{-2-i}\, \text{ if } R\in (\kappa^{-1}
      \rho^{1/d}, \kappa \rho)\\
      C_i  \rho^{\delta-2/d-i/d}\, \text{ if } R < \kappa^{-1}\rho^{1/d}.
    \end{cases}\]
   If in addition $d > 3$ and $\alpha$ is chosen sufficiently close to
   $1$, then we can even choose $\delta < 0$, i.e. in this case 
  $h$ decays faster than quadratically away
  from the singular rays in this case. 
\end{prop}
\begin{proof}
  The proof is similar to that of
  Proposition~\ref{prop:hest1}, but we scale our metric in a different way near
  the singular rays of $X_0$, and so we have several different regions
  to study separately. 

\bigskip
  {\bf Region I}: Consider the region where $R > \kappa\rho$, and
  $\rho\in (D/2, 2D)$ for some large $D$. Here we are uniformly far
  away from the singular rays. We study the scaled form
  $D^{-2}\omega$, in terms of rescaled coordinates 
  \[ \tilde{z} &= D^{-1}z \\
     \tilde{x} &= D^{-1}\cdot x,\]
   and we let $\tilde{r} = D^{-1}r$ (recall that here $D^{-1}\cdot x$
   is defined using the action with weights $w_i$). On this region, once $D$ is
   large enough, we have $\gamma_1=1, \gamma_2=0$, and so
  \[  D^{-2}\omega = \ddb\Big( |\tilde{z}|^2 + \tilde{r}^2\Big). \]
   In addition in terms of these coordinates $X_1$ has equation
   \[ D^{1-d}\tilde{z} + f(\tilde{x}) = 0, \]
   using that $f$ has degree $d$, and so $f(D\cdot \tilde{x}) = D^d
   f(\tilde{x})$. Note that $R\in (\kappa\rho, \rho)$, and $r$ is
   uniformly equivalent to $R$, and in addition $|\tilde{z}| < 2$. It
   follows that we are in essentially the same situation as in the
   proof of Proposition~\ref{prop:hest1}. The same arguments show that 
   \[ |\nabla^i h|_{D^{-2}\omega} \leq C_i D^{1-d}, \]
   and so on this region
   \[ |\nabla^i h|_\omega \leq C_i D^{1-d-i}. \]
   We can therefore choose any 
   $\delta$ such that $\delta-2 > 1-d$, i.e. $\delta > 3-d$. If $d >
   3$ we can choose $\delta < 0$, while if $d > 2$, then $3-d <
   2/d$ and so we can choose $\delta < 2/d$. 
   This implies the required estimates on this region. 

\bigskip
   {\bf Region II}: Suppose now that $R\in (K/2, 2K)$ for some $K <
   \kappa\rho$, but $K/2 > 2\rho^\alpha$. In addition let $\rho\in
   (D/2, 2D)$. In this case $\rho$ is comparable to $|z|$, and here we still have
   $\gamma_1=1, \gamma_2=0$. We assume that for some fixed $z_0$ we
   have $|z-z_0| < K$ (so that $z$ will be in a unit ball centered at
   $z_0$ after scaling). We now scale our form by $K$, and define
   \[\label{eq:scaleII}
     \tilde{z} = K^{-1}(z-z_0), \quad  \tilde{x} = K^{-1}\cdot x, \quad  \tilde{r} = K^{-1}r. \]
   In terms of these we have
   \[ K^{-2}\omega = \ddb\Big( |\tilde{z}|^2 + \tilde{r}^2\Big), \]
   and the equation of $X_1$ is
   \[ K^{-d}(K\tilde{z} + z_0) + f(\tilde{x}) = 0. \]
   In addition $|\tilde{z}| < 1$. We can still argue essentially like in
   Proposition~\ref{prop:hest1}, and now the errors we obtain are of
   order $K^{-d}D$ since $|\tilde{z}| < 1$ and $|z_0|\sim D$. 
The Ricci potential therefore satisfies
   \[ |\nabla^i h|_{\omega} \leq C_i DK^{-d-i}. \]
   Since $d > 2$ and $K > 4\rho^\alpha$, we have
   \[ DK^{2-d}K^{-2-i} < C D^{1 + \alpha(2-d)} K^{-2-i} < C
     \rho^{1+\alpha(2-d)} R^{-2-i}, \]
   for suitable $C$. We need to choose
   $\delta$ so that $1+\alpha(2-d) < \delta$. If
   $\alpha$ is sufficiently close to $1$, we can choose any $\delta >
   3-d$, which is the same constraint as in Region I.  If we merely
   have $\alpha > 1/d$, then we still have $1 + \alpha(2-d) < 2/d$, as
   required. 
   
   \bigskip
   {\bf Region III}: We now suppose that we are in the gluing region,
   so $R\in (K/2, 2K)$, but $K\in (\rho^\alpha, 2\rho^\alpha)$. Here the
   cutoff functions $\gamma_1, \gamma_2$ are not constant. Suppose
   that $\rho\in (D/2, 2D)$, and so $|z|$ is comparable to $D$. 
   We use the same scaling and change of variables \eqref{eq:scaleII}
   as in Region II.  We then
   have
   \[ K^{-2}\omega = \ddb\Big( |\tilde{z}|^2 + \gamma_1 \tilde{r}^2 +
     \gamma_2 K^{-2}|K\tilde{z} + z_0|^{2/d}
     \phi((K\tilde{z}+z_0)^{-1/d}K\cdot \tilde x)\Big). \] 
   and the cutoff functions $\gamma_1, \gamma_2$ have bounded derivatives
   in terms of these rescaled coordinates.
  The equation of $X_1$ is
   \[ K^{-d}(K\tilde{z} + z_0) + f(\tilde{x}) = 0, \]
   as above. We still want to compare the metric to
   \[ \ddb( |\tilde{z}|^2 + \tilde{r}^2) \]
   on $X_0$ which has equation $f(\tilde{x}) = 0$. 
   We use the estimate \eqref{eq:phiestimate} to get
   that in these coordinates 
   \[ \label{eq:vv1} \nabla^i\Big[ K^{-2}|K\tilde{z} + z_0|^{2/d}
     \phi((K\tilde{z}+z_0)^{-1/d}K\cdot \tilde x) - \tilde{r}^2 \Big]=
     O\left( (K^{-1}D^{1/d})^{2+c}\right). \]
   In other words, in the Ricci potential we will have the same error
   $K^{-d}D$ as in Region II from comparing the two equations, as well
   as a new error $(K^{-1}D^{1/d})^{2+c}$ from the error in the
   K\"ahler potential. Since now $K\sim D^\alpha$, this new term
   can be   estimated as follows:
   \[  \label{eq:e1} (K^{-1}D^{1/d})^{2+c} <
   CD^{\frac{2+c}{d}-c\alpha}K^{-2}. \]
   We need $\delta$ so that $\frac{2+c}{d} - c\alpha < \delta$. 
   Choosing $\alpha$ is sufficiently close to
   1, we can choose any $\delta > (2+c)/d - c$. If in addition $d >
   3$, then we have seen that we can choose $c > 1$, and so $(2+c)/d -
   c < 0$, and $\delta$ can be chosen negative. 
   In general if we only have $\alpha > 1/d$, then we still
   have $(2+c)/d - \alpha c < 2/d$, and so we can choose $\delta <
   2/d$. 

   \bigskip
   {\bf Region IV}: We now consider $R\in (K/2, 2K)$, but $K\in
   (\kappa^{-1}\rho^{1/d}, \rho^\alpha/2)$. Here $\gamma_1=0,
   \gamma_2=1$. We suppose that $\rho\in(D/2,2D)$ and so $|z|$
   is comparable to $D$. We scale in the same way as in Regions II, III, so
   we make the change of variables
   \eqref{eq:scaleII}. We have
   \[ K^{-2}\omega = \ddb\Big( |\tilde{z}|^2 +
   K^{-2}|K\tilde{z}+z_0|^{2/d} \phi((K\tilde{z}+z_0)^{-1/d}K\cdot \tilde x)\Big), \]
   and the equation of $X_1$ is still
   \[ K^{-d}(K\tilde{z} + z_0) + f(\tilde{x}) = 0. \]
   Instead of comparing $X_1$ to $X_0$,
   this time we want to compare $X_1$ to the product
   $\mathbf{C} \times V_1$, scaled suitably.  Note that $K^{-d}(K\tilde{z} + z_0)$ to
   leading order is $K^{-d}z_0$, and so we compare $X_1$ to the
   variety $\mathbf{C}\times V_{K^{-d}z_0}$ with equation
   \[ \label{eq:uu5} K^{-d}z_0 + f(\tilde{x}) = 0. \]
   The error introduced in passing from $X_1$ to this variety is of
   order $K^{1-d}$ (since $|\tilde{z}|<1$). 

  It remains to study the difference in the K\"ahler potentials. On
  the variety with equation \eqref{eq:uu5} we have the metric
  \[ \ddb( |\tilde{z}|^2 + K^{-2}|z_0|^{2/d} \phi(
  K|z_0|^{-1/d}\cdot\tilde{x})), \]
  and so we need to estimate the difference
  \[ E = K^{-2}|K\tilde{z}+z_0|^{2/d} \phi((K\tilde{z}+z_0)^{-1/d}K\cdot
  \tilde x) - K^{-2}|z_0|^{2/d} \phi(
  K|z_0|^{-1/d}\cdot\tilde{x}). \]
  Let us write $\phi = r^2 + \phi_{-c}$. By the homogeneity of $r$, we
  have
  \[ \label{eq:vv3} E = K^{-2}|K\tilde{z}+z_0|^{2/d} \phi_{-c}((K\tilde{z}+z_0)^{-1/d}K\cdot
  \tilde x) - K^{-2}|z_0|^{2/d} \phi_{-c}(
  K|z_0|^{-1/d}\cdot\tilde{x}). \]
  In addition since $|\tilde{z}| < 1, |z_0|\sim D$ and $K \ll D$, we have
  \[ \label{eq:vv4} K(K\tilde{z} + z_0)^{-1/d} = z_0^{-1/d}K\big(1 +
    O(KD^{-1})\big). \]\
  Using \eqref{eq:phiestimate} for $\phi_{-c}$, we therefore find that in
  the rescaled coordinates 
  \[ \label{eq:vv5} |\nabla^i E| < C_i (|z_0|^{-1/d}K)^{-2-c} D^{-1}K = O( K^{-1-c}
    D^{\frac{2+c}{d}-1}). \]
  Combining the error $K^{1-d}$ from changing the equation of the
  variety with this error in the K\"ahler potentials, we find that the
  Ricci potential is of order $K^{1-d} +
  K^{-1-c}D^{\frac{2+c}{d}-1}$. We therefore need to ensure that the
  choice of  $\delta$ satisfies
  \[ K^{1-d} +  K^{-1-c}D^{\frac{2+c}{d}-1} < C D^\delta K^{-2}. \]
  Suppose first that $ d> 3$, so that also $c > 1$. Then 
  \[ K^{1-d} = K^{3-d}K^{-2} < CD^{3/d-1}K^{-2},\]
 and also
  \[ K^{-1-c}D^{\frac{2+c}{d}-1} = (KD^{-1/d})^{1-c} D^{\frac{3}{d}-1}
    K^{-2}. \]
  Since $3/d-1 < 0$, we can choose $\delta < 0$ close to zero.

  If we only have $d > 2$, so $c > 0$, then
  \[ K^{-1-c}D^{\frac{2+c}{d}-1} = (KD^{-1/d})^{-c} (KD^{-1}) D^{2/d}
    K^{-2}. \]
  Since for some $C$ we have $C^{-1}D^{1/d} < K < CD$,   
  it follows that we can choose $\delta < 2/d$. 
  
    \bigskip
    {\bf Region V}: Finally we suppose that $R <
    2\kappa^{-1}\rho^{1/d}$. Again, let $z$ be very close to $z_0$,
    with $\rho$, and therefore $|z_0|$ comparable to $D$.
    We now scale by $|z_0|^{1/d}$, and so
    we change variables to 
    \[ \tilde{z} = z_0^{-1/d}(z-z_0), \quad \tilde{x} = z_0^{-1/d}\cdot x,\quad
      \tilde{r} = |z_0|^{-1/d}r, \]
    so that $|\tilde{z}|, \tilde{r} <C$.
    We have
    \[ |z_0|^{-2/d}\omega = \ddb\Big[ |\tilde{z}|^2 + |z_0|^{-2/d}\big|z_0^{1/d}\tilde{z}+z_0\big|^{2/d} \phi\big( z_0^{1/d}(z_0^{1/d}\tilde{z}+z_0)^{-1/d}\cdot\tilde{x}\big)\Big], \]
    and the equation of $X_1$ is
    \[ z_0^{1/d-1}\tilde{z} +1 + f(\tilde{x}) = 0. \]
    We want to compare this to $\mathbf{C}\times V_1$, with equation
    \[ 1 + f(\tilde{x}) = 0, \]
    and metric 
    \[ \ddb\big[ |\tilde{z}|^2 + \phi(\tilde x)\big]. \]
    The difference in the equations 
    results in an error of order $D^{1/d-1}$. For the K\"ahler
    potential note that
    \[ z_0^{1/d}(z_0^{1/d}\tilde{z}+z_0)^{-1/d} = 1 + O(D^{1/d-1}), \]
    and so 
  \[ |z_0|^{-2/d}\big|z_0^{1/d}\tilde{z}+z_0\big|^{2/d} \phi\big(
    z_0^{1/d}(z_0^{1/d}\tilde{z}+z_0)^{-1/d}\cdot\tilde{x}\big) -
    \phi(\tilde{x}) = O(D^{1/d-1}). \]
  In sum the Ricci potential is of order $D^{1/d-1}$, and we need
  $\delta$ that satisfies
  \[ D^{1/d-1} < CD^{\delta-2/d}, \]
    i.e. $\delta > 3/d - 1$. In particular if $d > 2$ we can choose $\delta < 2/d$,
    while if $d > 3$ we can choose $\delta < 0$. 
\end{proof}

\section{Weighted spaces on $X_1$}
As in the discussion before Proposition~\ref{prop:decayest2}, we
define a metric $\omega$ on all of $X_1$, which agrees with the form
defined in \eqref{eq:approx1} for sufficiently large $\rho$. Our eventual
goal is to perturb the metric $\omega$ to a Calabi-Yau metric on the
set $\{\rho > A\}$ for sufficiently large $A$, at which point we will
be able to apply Hein~\cite[Proposition 4.1]{HeinThesis} to construct
a global Calabi-Yau metric on $X_1$. The main difficulty 
is to invert the linearized operator, which is the Laplacian
on $(X_1, \omega)$.

The analysis of the Laplacian on asymptotically conical spaces has
been studied extensively (see e.g. Lockhart-McOwen~\cite{LM85}),
and was used in the work of
Conlon-Hein~\cite{CH13} discussed in Section~\ref{sec:smoothing1} above.
The difference is that we now need to invert the Laplacian in
a more complicated weighted space that accounts for the singular rays
in the tangent cone at infinity. 
This almost fits into the framework of QAC
spaces studied by Degeratu-Mazzeo~\cite{DM14}, which in turn
generalizes the work of Joyce~\cite{Joy00} on QALE manifolds, but
unfortunately those results cannot be applied directly in our
setting. The issue is that at distance $D$ along the
singular rays our metric $\omega$ is modeled on a product of
$\mathbf{C}$ with a scaling of the metric $\omega_{V_1}$ by a factor of
$D^{2/d}$. In the QAC or QALE geometries there is no such additional
scale factor on the geometry ``transverse'' to the singular rays. Very
recently, Conlon-Rochon~\cite{CR17} have extended the
techniques of Degeratu-Mazzeo to such a more
general setup, however an additional difficulty in our setting is
that we need to invert the Laplacian on functions that decay more
slowly than what they consider. Because of this, we follow a different
approach, which has the added advantage that
our method applies with essentially no changes
to constructing certain Calabi-Yau
metrics in a neighborhood of an isolated singularity, 
as we will see in Section~\ref{sec:singularity}.

We first define weighted spaces $C^{k,\alpha}_{\delta,
  \tau}(X_1,\omega)$, and then define
$C^{k,\alpha}_{\delta,\tau}(\rho^{-1}[A,\infty),\omega)$ by
restricting functions to the set where $\rho \geq A$. The norm of
a function on $\rho^{-1}[A,\infty)$ is defined to be the infimum of
the corresponding norms of its extensions to $X_1$. 

On the set where $\rho < P$ for some fixed large $P$ (we will have $P < A$)
we use the usual
H\"older norms. When $\rho > P$ then we define the weighted spaces in
terms of the functions $\rho$ (which is essentially the radial
distance), and $R$ (which controls the distance from the singular
rays). To make the analogy with Degeratu-Mazzeo's weighted spaces we
define a smooth function $w$ satisfying  
\[ w = \begin{cases} 1 \quad &\text{ if } R > 2\kappa\rho, \\
  R / (\kappa\rho) \quad &\text{ if } R\in (\kappa^{-1}\rho^{1/d},
  \kappa\rho),\\
  \kappa^{-2}\rho^{1/d-1} \quad &\text{ if } R<
  \frac{1}{2}\kappa^{-1}\rho^{1/d}, \end{cases} \]
for the same $\kappa$ as in Proposition~\ref{prop:decayest2}. 

Similarly to Degeratu-Mazzeo, we define the H\"older seminorm
\[ [T]_{0,\gamma} = \sup_{\rho(z) > K} \rho(z)^{\gamma}
w(z)^{\gamma} \sup_{z'\ne z, z'\in B(z,c)} \frac{|T(z) -
  T(z')|}{d(z,z')^\gamma}. \]
Here $c > 0$ is such that $X_1$ has bounded geometry on balls of
radius $c > 0$ and these balls are geodesically convex. 
$T$ could be a tensor, in which case we compare
$T(z)$ with $T(z')$ using parallel transport along a geodesic. 

The weighted norm of $f$ is then defined by
\[ \label{eq:weightedHolder}
 \Vert f\Vert_{C^{k,\alpha}_{\delta, \tau}} &= \Vert
f\Vert_{C^{k,\alpha}(\rho < 2P)} + \sum_{j=0}^k \sup_{\rho(z) > P}
\rho^{-\delta+j}w^{-\tau+j}|\nabla^j f| \\
&\quad + [\rho^{-\delta +k}w^{-\tau +k} \nabla^k f]_{0,\alpha}. \]
A more concise way to express these weighted norms is in terms of the
conformal scaling $\rho^{-2}w^{-2}\omega$. 
More precisely for this we should replace $\rho$ by a smoothing of
$\max\{1, \rho\}$. 
In terms of these weight functions our weighted norms could be defined
equivalently as
\[ \Vert f\Vert_{C^{k,\alpha}_{\delta,\tau}} = \Vert \rho^{-\delta}
  w^{-\tau} f\Vert_{C^{k,\alpha}_{\rho^{-2}w^{-2} \omega}}. \]
The estimate in Proposition~\ref{prop:decayest2} can now be stated as
saying that the Ricci potential $h$ is in the weighted space
$C^{k,\alpha}_{\delta-2,-2}$, with $\delta$ as in the Proposition. 
Moreover since $w \leq 1$, we also have $h \in
C^{k,\alpha}_{\delta-2, \tau-2}$ for any $\tau < 0$.

In the following propositions we compare the geometry of $(X_1,\omega)$ to
suitable model spaces in different regions. This will be used to study
the Laplacian on $(X_1,\omega)$ in Section~\ref{sec:Laplaceinverse}
below. We will write $g, g_{X_0}$ for the Riemannian metrics given by
$\omega, \omega_{X_0}$ respectively. 

Let us first consider the region
\[ \mathcal{U} = \{\rho > A, R > \Lambda \rho^{1/d}\}\cap X_1,\]
for large $A, \Lambda$, and let
\[  G: \mathcal{U} \to X_0 \]
be the nearest point projection inside $\mathbf{C}^{n+1}$, for the
cone metric $\ddb(|z|^2 + R^2)$. We have
\[ G(z,x) = (z,x'), \]
where $x'\in V_0$ is the nearest point projection of $x\in
\mathbf{C}^n$ under the metric $\ddb R^2$.
\begin{prop}\label{prop:R>L}
  Given any $\epsilon > 0$ we can choose $\Lambda >
  \Lambda(\epsilon)$, and $A > A(\epsilon)$ sufficiently large
  so that on $\mathcal{U}$ we have
  \[  |\nabla^i (G^*g_{X_0} - g)|_{g} < \epsilon
    w^{-i}\rho^{-i}, \]
  for $i\leq k+1$. 
  In particular in terms of our weighted spaces we have
  \[  \Vert G^*g_{X_0} - g\Vert_{C^{k,\alpha}_{0,0}} <
    \epsilon. \]
\end{prop}
\begin{proof}
  The proof is very similar to the analysis of regions I,II, III, IV in the proof
  of Proposition~\ref{prop:decayest2}. Let us first consider region
  I. Suppose that $\rho\in (D/2, 2D)$, with $D > A$, and $R > \kappa
  \rho$.  Here we have $w\sim 1$, the notation $a\sim b$ meaning that
  $C^{-1} < a/b < C$ for a constant $C$.  The estimate that we need to show
  is equivalent to
  \[ |D^{-2}\nabla^i(G^*g_{X_0} - g)|_{D^{-2}g} <
    \epsilon. \]
  We can work in the rescaled coordinates $\tilde{z}, \tilde{x}$ as in
  Proposition~\ref{prop:decayest2}. In these coordinates $X_1$ is
  given by the equation $D^{1-d}\tilde{z} + f(\tilde{x}) = 0$, and
  $X_0$ by the equation $f(\tilde{x})=0$. In
  Proposition~\ref{prop:decayest2} we were only interested in
  estimating the Ricci potential, and so we compared the two
  hypersurfaces in local holomorphic charts (see the proof of
  Proposition~\ref{prop:hest1}). Now, however, we want a more global
  comparison, using the projection map $G$ which is not holomorphic. 
This is closer to the  approach taken by
Conlon-Hein~\cite{CH13}. The end result is the same, since in local
charts the difference between $G$ and the identity map is of order
$D^{1-d}$. Since $d > 1$ we can ensure that this is less than
$\epsilon$ by taking $A$ large. 

  Regions II, III follow similar calculations to the proof of
  Proposition~\ref{prop:decayest2}. Let us therefore consider region
  IV, which is a little different, since previously we compared this
  region to $\mathbf{C}\times V_1$, whereas now we are comparing it to
  $X_0 = \mathbf{C}\times V_0$. Suppose that $\rho \in (D/2, 2D)$, and
  $R\in (K/2, 2K)$ with $\Lambda\rho^{1/d} < K < \rho^\alpha/2$. In
  this region $\rho$ is comparable to $|z|$, and we suppose that $z$
  is close to some $z_0$, with $|z_0| \in (D/4, 4D)$. Introducing
  coordinates $\tilde{z}, \tilde{x}$ as before, the equation for $X_1$
  is
  \[ K^{-d}(K\tilde z + z_0) + f(\tilde{x}) = 0, \]
  while the equation for $X_0$ is $f(\tilde{x}) = 0$. As long as
  $|\tilde{z}| < 1$, the error introduced by orthogonal projection is
  of order $K^{1-d}$, which can be made arbitrarily small by choosing
  $D$ large. 

  The two K\"ahler potentials that we need to compare are
  \[ \ddb\left(|\tilde{z}|^2 + K^{-2}|K\tilde{z} + z_0|^{2/d}
      \phi\big( (K\tilde{z} + z_0)^{-1/d} K\cdot
      \tilde{x}\big)\right), \]
  and
  \[ \ddb \left( |\tilde{z}|^2 + \tilde{r}^2\right). \]
  If $|\tilde{z}| < 1$ and $|z_0|\sim D$, by the estimate
  \eqref{eq:phiestimate} the difference between the K\"ahler
  potentials is of order
  \[ (KD^{-1/d})^{-c-2} < C \Lambda^{-c-2}. \]
  By choosing sufficiently large $\Lambda$, we can ensure that this is
  less than $\epsilon$. 
\end{proof}

Next, we focus on the region where $\rho > A$, but $R <
\Lambda\rho^{1/d}$. Fix $z_0\in \mathbf{C}$ with $|z_0| > A$, and a
large constant $B$. 
Consider the region $\mathcal{V}\subset X_1$ given by points $(z,x)$
satisfying $|z-z_0| < B|z_0|^{1/d}$ and where $R <
\Lambda\rho^{1/d}$. Let us define new coordinates
\[ \hat{x} = z_0^{-1/d}\cdot x, \quad \hat{z} =
  z_0^{-1/d}(z-z_0), \]
and let $\hat{R} = |z_0|^{-1/d}R$. In addition we let $\hat\zeta =
\max\{1, \hat{R}\}$. Note that $(\hat{z}, \hat{x})$
satisfy the equation
\[ \label{eq:tildezx} \hat{z_0}^{1/d-1}\hat{z} + 1 + f(\hat{x}) =
  0, \]
and $|\hat{z}| < B$, $|\hat{R}| < C\Lambda$ for some fixed constant $C$
(since $\rho$ is comparable to $|z_0|$). 
We define the map
\[ H : \mathcal{V} \to \mathbf{C}\times V_1 \]
by letting $H(\hat{z}, \hat{x}) = (\hat{z}, \hat{x}')$, where
$\hat{x}'$ is the nearest point projection from the solutions of
\eqref{eq:tildezx} to solutions of $1 + f(\hat{x})=0$.
\begin{prop}\label{prop:R<L}
  Given $\epsilon, \Lambda > 0$ if $A >
  A(\epsilon, \Lambda, B)$, then we have
  \[ |\nabla^i(H^*g_{\mathbf{C}\times V_1} -
    |z_0|^{-2/d}g)|_{|z_0|^{-2/d}g} < \epsilon \hat{\zeta}^{-i}, \]
   for $i \leq k+1$. 
  In terms of weighted spaces we then have
  \[ \Vert |z_0|^{2/d}H^*g_{\mathbf{C}\times V_1} - g\Vert_{C^{k,\alpha}_{0,0}}
    < \epsilon. \]
\end{prop}
\begin{proof}
  This follows the analysis of our metric in regions IV, V in
  Proposition~\ref{prop:decayest2}. Let us focus on region IV, which
  is the more complicated one. As before, we have $\rho\in (D/2, 2D)$,
  $R\in (K/2,2K)$, such that $\kappa^{-1}D^{1/d} < K < \Lambda
  D^{1/d}$. Here $|z_0|$ is comparable to $\rho$, and  $\hat{\zeta}$ is
  comparable to $\hat{R} = |z_0|^{-1/d}R$, so $\hat{\zeta} \sim
  |z_0|^{-1/d}K$. The  estimate we need to show is therefore
  \[ |\nabla^i(|z_0|^{2/d}K^{-2}H^*g_{\mathbf{C}\times V_1} -
    K^{-2}g)|_{K^{-2}g} < \epsilon. \]

  We introduce the coordinates $\tilde{z}, \tilde{x}$ as in
  Proposition~\ref{prop:decayest2}, 
  which in terms of $\hat{x}, \hat{z}$ are 
  \[ \tilde{x} = K^{-1}z_0^{1/d}\cdot \hat{x}, \quad \tilde{z} =
    K^{-1}z_0^{1/d}\hat{z}. \]
  In these coordinates $X_1$ is given by
  \[ K^{-d}(K\tilde{z} + z_0) + f(\tilde x) = 0. \]
  We need to compare the metric $\omega$ on $X_1$ with the product
  metric $\mathbf{C}\times V_{K^{-d}z_0}$ on the hypersurface with
  equation
  \[ K^{-d}z_0 + f(\tilde{x}) = 0, \]
  under the closest point projection map. 
  By the same calculations as in the analysis of region IV in
  Proposition~\ref{prop:decayest2} we find that the error between the
  two metrics is of order $BD^{1/d-1}$, and this can be made
  arbitrarily small by taking $D$ large, since $d > 1$. 
\end{proof}

We will need an extension operator
$C^{0,\alpha}_{\delta,\tau}(\rho^{-1}[A, \infty), \omega) \to
C^{0,\alpha}_{\delta,\tau}(X_1,\omega)$. More sophisticated methods
could be used to deal with the $C^{k,\alpha}$ spaces for $k>0$ as
well (see e.g. Seeley~\cite{Seeley}), but for our purposes $k=0$
suffices.
\begin{prop}\label{prop:extensionop}
  For sufficiently large $A$, there is a linear extension operator
  \[ E : C^{0,\alpha}_{\delta,\tau}(\rho^{-1}[A, \infty), \omega) \to
    C^{0,\alpha}_{\delta,\tau}(X_1,\omega), \]
  whose norm is bounded independently of the choice of $A$. 
\end{prop}
\begin{proof}
  The basic observation is that a $C^{0,\alpha}$ function $f$ on a half space
  $\mathbf{R}^n_+ = \{x_n \geq 0\} \subset \mathbf{R}^n$ can be
  extended to a function on $\mathbf{R}^n$ by reflection, while preserving the
  $C^{0,\alpha}$-norm. In addition, multiplying by a cutoff function,
  we can define an extension $E(f)$ supported in the set $\{x_n >
  -1\}$,  such that $\Vert E(f)\Vert_{C^{0,\alpha}} \leq C\Vert
  f\Vert_{C^{0,\alpha}(\mathbf{R}^n_+)}$. The same applies if instead
  of a half space, $f$ is defined on the set $\{ x_n \geq
  F(x_1,\ldots,x_{n-1})\}$, where $F$ is a $C^{0,\alpha}$ function on
  $\mathbf{R}^{n-1}$, and the norm of the extension operator will then
  also depend on the $C^{0,\alpha}$-norm of $F$. 
  
  We can globalize this construction to extending functions from
  $\rho^{-1}[A,\infty)$ to $X_1$, using that near each point in
  $\rho^{-1}(A)$ we control the geometry of $\omega$ at suitable
  scales. To do this, let $x\in \rho^{-1}(A)$ and consider the ball
  $B(x,r_x)$, where $r_x$ is defined   as follows:
  \[ r_x = \begin{cases} \kappa A/10, &\quad \text{ if } R >
      \kappa\rho, \\
      R/10, &\quad \text{ if }\kappa^{-1}\rho^{1/d} < R < \kappa\rho,
      \\
      A^{1/d}, &\quad \text{ if }R < \kappa^{-1}\rho^{1/d}.
    \end{cases} \]
  These radii are chosen so that on $B(x,r_x)$ we have $\rho w \sim
  r_x$. Let us analyze these balls at the scale $r_x$. 
  \begin{itemize}
    \item If at $x$ we have $R>\kappa \rho$, then from the analysis of
      Region I in Proposition~\ref{prop:decayest2}
      we know that the scaled down metric $A^{-2}\omega$ on $B(x,r_x)$ 
      converges to the cone metric on $X_0$ as $A\to \infty$ on a
      ball of radius $A^{-1}r_x = \kappa/10$ in $X_0$. Moreover this
      ball is centered at a point $\tilde{x}$ satisfying
      $\rho(\tilde{x})=1$, $R(\tilde{x}) > \kappa$. We therefore
      control the geometry of the boundary $\rho^{-1}(A)$ uniformly,
      and we can extend functions from $\rho^{-1}([A,\infty)) \cap
      B(x,r_x)$ to $B(x,r_x)$ as above. Note in addition that by the
      definition of the weighted norms we have
      \[  \Vert f\Vert_{C^{0,\alpha}_{\delta, \tau}(B(x,r_x), \omega)} \sim
        A^\delta\Vert f\Vert_{C^{0,\alpha}(B(x,r_x), A^{-2}\omega)}, \]
      and so we can control the weighted H\"older norm of the extension.
    \item If at $x$ we have $\kappa^{-1}\rho^{1/d} < R < \kappa\rho$,
      and $R\in (K/2, 2K)$, then from the analysis of Regions II,III,
      IV in Proposition~\ref{prop:decayest2} we know that the scaled
      down metric $K^{-2}\omega$ approaches the model metric on
      $\mathbf{C}\times V_s$ for some $s$ with $|s| \leq 1$. In each
      of these spaces we control the geometry of $\rho^{-1}(A)$, and
      so we can define an extension map. The weighted H\"older norms
      here are related to the norms with respect to the rescaled
      metric by
      \[  \Vert f\Vert_{C^{0,\alpha}_{\delta, \tau}(B(x,r_x), \omega)} \sim
        A^{\delta-\tau}K^\tau\Vert f\Vert_{C^{0,\alpha}(B(x,r_x),
          K^{-2}\omega)}. \]
      \item In a similar way, if $x$ is in the third region $R <
        \kappa^{-1}\rho^{1/d}$, then the rescaled metric
        $A^{-2/d}\omega$ approaches the product metric
        $\mathbf{C}\times V_1$, and again we can define an extension
        map.  The relation between the weighted H\"older norm, and the
        H\"older norm with respect to the rescaled metric is
        \[ \Vert f\Vert_{C^{0,\alpha}_{\delta, \tau}(B(x,r_x), \omega)} \sim
        A^{\delta-\tau + \tau/d}\Vert f\Vert_{C^{0,\alpha}(B(x,r_x),
          A^{-2/d}\omega)}. \]
    \end{itemize}
   Finally, to globalize these local extensions we can use cutoff functions. 
\end{proof}

The next result shows that the tangent cone of $(X_1,\omega)$ at
infinity is $X_0 = \mathbf{C}\times V_0$.
\begin{prop}\label{prop:tancone}
  Let $\epsilon > 0$. If $D$ is sufficiently large, then the
  Gromov-Hausdorff distance between the regions defined by 
  $\rho\in (D/2, 2D)$ in $(X_1, \omega)$, and in $X_0=\mathbf{C}\times
  V_0$  with the product  metric $\omega_0$,  is less than $D\epsilon$. 
\end{prop}
\begin{proof}
 It will be useful to
introduce the notation 
\[ S_{\Lambda} = \{ x\in X_1\,:\, R < \Lambda\rho^{1/d} \}, \]
which we should think of as a region near the singular rays in $X_0$.
  Let us denote by $X_1^D, X_0^D$ the two annular regions that we are
  considering, with their metrics scaled down by a factor of $D$. 
  Our goal is to define a map $G : X_1^D \to X_0^D$, such
  that for sufficiently large $D$ we have 
  \[ \label{eq:Gest1} |d(G(x_1),G(x_2)) - d(x_1,x_2)| < \epsilon, \]
  and the image $G(X_1^D)$ is $\epsilon$-dense in $X_0^D$. 

  Let $\Lambda$ be a large constant. On the set $X_1^D\setminus
  S_\Lambda$ we define $G$ by the nearest point projection, as in
  Proposition~\ref{prop:R>L}, 
  while on the set $X_1^D\cap S_\Lambda$ we define $G$ by
  projection onto the $z$-axis. From Proposition~\ref{prop:R>L} we
  know that on $X_1^D\setminus S_\Lambda$ the
  rescaled metrics $D^{-2}\omega$ and $D^{-2}G^*\omega_0$ can be made
  arbitrarily close by choosing $\Lambda,D$ large. We
  can therefore assume that for any curve $\gamma$ in this region
  \[ \label{eq:lg2} |\mathrm{length}_{D^{-2}\omega}(\gamma) -
    \mathrm{length}_{D^{-2}G^*\omega_0}(\gamma)| < \epsilon, \]
  and in particular if $x_1,x_2$ are in this region, then 
  \[ d_{X_1^D}(x_1,x_2) < d_{X_0^D}(G(x_1),G(x_2)) + \epsilon. \]
  Note that the reverse inequality is not yet clear since the shortest
  curve between $x_1,x_2$ in $X_1^D$ may pass through $S_\Lambda$. 

  At the same time, on $X_1^D\cap S_{2\Lambda}$,
  Proposition~\ref{prop:R<L} shows that after suitable scalings we can
  approximate $X_1^D$ by the product metric $\mathbf{C}\times V_1$. In
  particular, choosing first $\Lambda$ and then $D$ sufficiently
  large,   we can assume that the projection $\pi  :
  X_1^D\cap S_{2\Lambda}\to \mathbf{C}$ onto the $\mathbf{C}$ factor satisfies $|d\pi|
  < 1 + \epsilon$. It follows from this that for any curve $\gamma$ in
  this region 
  \[  \label{eq:lg1} \mathrm{length}_{X_1^D} (\gamma) \geq (1-\epsilon)
    \mathrm{length}_{\mathbf{C}}(\pi\circ\gamma). \]
  Note that if $x_1, x_2 \in S_\Lambda$, then the
  shortest curve between them in $X_1^D$ will remain in the region
  $S_{2\Lambda}$ since on the ``annular'' region
  $S_{2\Lambda}\setminus S_\Lambda$ the metric can be made
  arbitrarily close to the cone $X_0^D$ (if $\Lambda$ and $D$ are sufficiently
  large). It follows using \eqref{eq:lg1} that
  \[ d_{X_1^D}(x_1,x_2) > d_{\mathbf{C}}(\pi(x_1), \pi(x_2)) -
    \epsilon. \]

  In addition if we fix a reference point $o\in V_1$, then for any
  $(p,z)\in X_1^D\cap S_\Lambda$ we have
  \[ \label{eq:Gest2} d_{X_1^D}( (p,z), (o,z) ) < C\Lambda D^{1/d-1} < \epsilon, \]
  for some constant $C$ (recall that we choose $D$ after choosing
  $\Lambda$).  
  It then follows that if $(p_1,z_1), (p_2, z_2)$ are two points in
  this region, we have 
  \[ d_{X_1^D}((p_1,z_1) , (p_2,z_2)) < |z_1-z_2| + \epsilon, \]
  and similarly
  \[ d_{X_0^D}((p_1,z_1) , (p_2,z_2)) < |z_1-z_2| + \epsilon. \]
  We have therefore shown that if $x_1,x_2 \in S_\Lambda$, then 
  \[ |d_{X_1^D}(x_1,x_2) - d_{X_0^D}(G(x_1), G(x_2))| < \epsilon. \]
  
  It remains to bound the distance from below between points
  $x_1,x_2\in X_1^D\setminus S_\Lambda$. Let $\gamma$ be the shortest
  curve from $x_1$ to $x_2$. We only have to deal with the possibility that
  this shortest curve enters the closure of the region $S_\Lambda$. 
  Let $x_1'$ be the first, and $x_2'$ be the
  last point along the curve in this region. Then by the observations
  above, the segment of $\gamma$ joining $x_1',x_2'$ must lie in
  $S_{2\Lambda}$,  and therefore we have
  \[ d_{X_1^D}(x_1', x_2') > d_{\mathbf{C}}(\pi(x_1'), \pi(x_2')) -
    \epsilon > d_{X_0^D}(G(x_1'), G(x_2')) - 2\epsilon. \]
  Then by the triangle inequality and the estimate \eqref{eq:lg2} we
  get
  \[ d_{X_1^D}(x_1,x_2) > d_{X_0^D}(G(x_1),(x_2)) - \epsilon \]
  for a larger value of $\epsilon$. This shows \eqref{eq:Gest1}. The
  fact that $G(X_1^D)$ is $\epsilon$-dense in $X_0^D$ follows from the
  inequality analogous to \eqref{eq:Gest2} for $X_0^D$. 
\end{proof}

It follows from this result in particular that if $o\in X_1$ is a
fixed basepoint, then the distance function $d(o,\cdot)$ is uniformly
equivalent to $\rho$. 
Later on we will also need that $(X_1,\omega)$ satisfies the following
``relatively connected annuli'', or RCA, condition. This is an easy
consequence of the fact that the tangent cone at infinity of
$(X_1,\omega)$ is a metric cone over a compact connected length space. 
\begin{prop}\label{prop:RCA}
  For sufficiently large $D$, any two points $x_1,x_2\in X_1$ with
  $d(o,x_i)=D$ can be joined by a curve of length at
  most $CD$, lying in the annulus $B(o, CD) \setminus B(o,C^{-1}D)$,
  for a uniform constant $C$.   
\end{prop}

\section{The Laplacian on the model spaces}\label{sec:modelspace}
Here we consider two model problems, namely the Laplacians on the cone
$X_0 = \mathbf{C}\times V_0$, 
and on the product $\mathbf{C}\times V_1$. These will be the
building blocks for inverting the Laplacian on $X_1$ in
Section~\ref{sec:Laplaceinverse}. As a preliminary
step we consider the Laplacian on the product $\mathbf{C} \times
V_0$, but using weighted spaces which involve a weight function only
in the $V_0$ factor. 

\subsection{The model space $\mathbf{C}\times V_0$}
Consider the product metric $g = g_{\mathbf{C}} + g_{V_0}$,
and let $r$ be the distance function on $V_0$ as before. We define
weighted spaces on $\mathbf{C}\times V_0$ in the usual way,
analogously to \eqref{eq:weightedHolder} with the
weight function given by $r$. A more concise definition is obtained by
conformally scaling $g$ to $r^{-2}g$. In terms of this scaled metric
our weighted spaces can be written as
\[ C^{k,\alpha}_\tau = r^\tau C^{k,\alpha}_{r^{-2}g}, \]
with corresponding norm given by
\[ \Vert f\Vert_{C^{k,\alpha}_\tau} = \Vert r^{-\tau}
  f\Vert_{C^{k,\alpha}_{r^{-2}g}}. \]

Our goal is to show that the kernel of the Laplacian on
$(\mathbf{C}\times V_0, g)$ is trivial in the weighted H\"older space
$C^{k,\alpha}_\tau$, for $\tau\in (4-2n, 0)$. Note that the (real) dimension
of $V_0$ is $m=2n-2$, so $\tau\in (2-m, 0)$ is the usual ``good''
range of weights for the Laplacian on $V_0$. We use the Fourier
transform in the $\mathbf{C}$ direction in a similar way to what was
done by Walpuski~\cite{Wal13} which in turn is based on
Brendle~\cite{Bre03} (see also Mazzeo-Pacard~\cite{MP96}), 
although the details will be slightly different. 

For any function $\chi$ on $\mathbf{C}\times V_0$ 
with sufficient decay in the $\mathbf{C}$ direction we define the
Fourier transform
\[ \hat{\chi}(\xi, x) = \int_{\mathbf{C}} \chi(z, x) e^{-\sqrt{-1}\xi\cdot
    z}\, dz, \]
where we think of $\mathbf{C}$ as $\mathbf{R}^2$. 

\begin{prop}\label{prop:solve1}
  Let $\chi$ be a smooth function on $\mathbf{C}\times V_0$, such that
  $\hat{\chi}$ has compact support away from $\{0\} \times V_0$. 
  In particular $\chi$ is supported
  in $\mathbf{C}\times K$ for a compact $K\subset V_0$. 
  We can then find $f$ solving
  $\Delta f = \chi$, and moreover
  \begin{enumerate}
  \item $f \in C^{k,\alpha}_\tau$ for all $\tau\in (2-m,0)$,
  \item In addition $f$ decays exponentially in the $\mathbf{C}$ direction. More
    precisely for any $a > 0$ there is a constant $C$ such that $\Vert
    f \Vert_{C^{k,\alpha}_\tau(|z| > A)}
    < C(1 + A)^{-a}$ for any $A > 0$ and $\tau\in (2-m,0)$.  
  \end{enumerate}
\end{prop}
\begin{proof}
  Taking the Fourier transform, the equation that we need to solve is
  \[ \label{eq:FTeqn} \Delta_{V_0} \hat{f}(\xi, x) - |\xi|^2 \hat{f}(\xi, x) =
    \hat{\chi}(\xi, x). \]
  For each fixed $\xi \ne 0$ we can solve this in polar coordinates on
  $V_0$, using the spectral decomposition for the link of $V_0$. 
  It reduces to analyzing ODEs of the form
  \[ \label{eq:ODEpolar}\partial_r^2 \hat{f} + \frac{m-1}{r} \partial_r \hat{f} +
    \frac{1}{r^2} \lambda \hat{f} - |\xi|^2 \hat{f} =
    \hat{\chi}_\lambda, \]
  where $\lambda \leq 0$ are the eigenvalues of the Laplacian on the
  link of $V_0$. Note that the $\hat{\chi}_\lambda$ have compact
  support in $r$. For $\xi\ne 0$ the ODE has a fundamental solution
  which decays
  exponentially as $r\to\infty$, and is bounded near $r=0$. Using this
  we obtain solutions $\hat{f}(\xi, x)$ of \eqref{eq:FTeqn} decaying
  exponentially as $r\to\infty$, and in particular
  satisfying bounds $\Vert
  \hat{f}(\xi,\cdot)\Vert_{C^0_\tau} \leq C\Vert \hat{\chi}(\xi,
  \cdot)\Vert_{C^0}$, where $C$ is uniform as long as $\xi$ is
  bounded. 

  Our results follow by taking the inverse Fourier transform of
  $\hat{f}$. Item (1) above follows from the properties of the ODE
  solutions. Item (2) follows from the fact that if $\hat\chi$ has compact
  support away from $\{0\}\times V_0$, then $\hat{f}$ is smooth when
  viewed as a function $\mathbf{R}^2\to C^0_\tau(V_0)$. 
  This can be seen by differentiating
  \eqref{eq:FTeqn} with respect to $\xi$. We find that any derivative
  $\partial^l_\xi\hat{f}$ satisfies an equation of the
  form
  \[ \label{eq:aaa1} \Delta_{V_0} \partial^l_\xi \hat{f}(\xi, x) -
    |\xi|^2\partial^l_\xi =
    \partial^l_\xi\hat{\chi}(\xi,x) + \sum_{|i| < |l|} a_i(\xi) \partial^i_\xi
    \hat{f}(\xi,x) \]
  in terms of lower order derivatives. Inductively we find that each
  partial derivative $\partial^l_\xi\hat{f}(\xi,x)$ is bounded near $r=0$
  and decays exponentially as $r\to \infty$, using that $\hat{f}(\xi,x)$ is
  identically zero for $\xi$ in a neighborhood of 0. 

  Taking the invese Fourier transform we 
  obtain exponential decay in $|z|$ of the $C^0_\tau$ norms of $f$ in
  the vertical slices $\{z\}\times V_0$. The decay of the 
  derivatives of $f$ is obtained from this using Schauder estimates since
  $\chi$ is compactly supported in the $V_0$ direction.
\end{proof}

\begin{cor}\label{cor:ker0}
  Suppose that $\Delta f =0$, for some $f\in
  C^{k,\alpha}_\tau(\mathbf{C}\times V_0)$ with
  $\tau\in (2-m,0)$. Then $f=0$. 
\end{cor}
\begin{proof}
  We will argue by taking the distributional Fourier transform of $f$,
  and showing that it is supported at the origin. It will be more
  convenient to use weighted $L^2$-spaces than the H\"older
  spaces. For this let $\tau_1,\tau_2$ be such that $\tau_1 < \tau <
  \tau_2$, and define a weight function $\sigma$ on $V_0$ such that
  $\sigma = r^{\tau_2}$ for large $r$, and $\sigma = r^{\tau_1}$ for
  small $r$. We define the weighted $L^2_\sigma$ space on $V_0$ using the norm
  \[ \Vert F \Vert_{L^2_\sigma}^2 = \int_{V_0} |F|^2 \sigma^{-2}
    r^{-m}\,dV. \]
  Our choice of $\tau_1,\tau_2$ ensures that we can view $f$ as a bounded map
  \[ f : \mathbf{C} \to L^2_\sigma(V_0). \]
  In addition the dual space under the $L^2$ pairing is $L^2_{\sigma'}$, where
  $\sigma'=\sigma^{-1}r^{-2m}$. 
  Let $\hat{f}$ be the Fourier transform of $f$ in the
  $\mathbf{C}$-direction as above, so now $\hat{f}$ is a distribution
  valued in $L^2_\sigma$. If $g$ is a
  smooth map $\mathbf{C}\to L^2_{\sigma'}$ of compact support, 
  then the pairing $\hat{f}(g)$ is
  defined as
  \[ \hat{f}(g) = \int_\mathbf{C} \langle f, \hat{g}\rangle\, dz. \]

  We first claim that $\hat{f}$ is a distribution of finite order (in
  fact order at most 4). For this suppose that $g : \mathbf{C}\to
  L^2_{\sigma'}$ has compact support $K$, and
  satisfies $|\nabla^i g| \leq A$ for $i\leq 4$, where
  the derivatives are in the $\mathbf{C}$-direction. Then the usual
  integration by parts argument shows that $\Vert\hat{g}(\xi)\Vert
  \leq C_{K,A}(1+|\xi|)^{-4}$, for a constant depending on $K,A$, and so
  \[ |\hat{f}(g)| \leq \int_\mathbf{C} \Vert
    f(z,\cdot)\Vert_{L^2_\sigma} \Vert \hat{g}(z)\Vert_{L^2_{\sigma'}}\, dz \leq
    C_{K,A}, \]
  since $f$ is bounded. 

  Next we claim that $\hat{f}$ is supported at the origin, i.e. if
  $g$ is smooth (in the $\mathbf{C}$-directions) and
  has compact support away from the origin, then
  $\hat{f}(g)=0$. We can approximate $g$ in the space $C^4(\mathbf{C},
  L^2_{\sigma'})$ with smooth functions that have compact support away
  from $\{0\}\times V_0$, and so we assume that $g$ is of this
  type. We can then apply Proposition~\ref{prop:solve1} to the
  function $\chi = \hat{g}$, and so we find $h$ on $\mathbf{C}\times
  V_0$ satisfying $\Delta h =
  \hat{g}$, such that $h\in C_\tau^{k,\alpha}$, and $h$ decays
  exponentially in the $\mathbf{C}$-direction. It follows that
  \[ \hat{f}(g) &= \int_{\mathbf{C}\times V_0} f\hat{g}\,dV \\
    &= \int_{\mathbf{C}\times V_0} f\Delta h\,dV \\
    &= \int_{\mathbf{C}\times V_0} h\Delta f\,dV = 0, \]
  where the integration by parts is justified by the decay properties
  of $h$. 

  Since $\hat{f}$ is supported at the origin, it follows that
  $\hat{f}$ is a linear combination of derivatives of delta functions
  at the origin, with coefficients in $L^2_\sigma$. This follows the
  usual argument from the scalar case (see
  Rudin~\cite [Theorem 6.25]{Rudin}). Therefore we have
  \[ f(z) = \sum_{i,j=1}^l z^i\bar{z}^j f_{ij}, \]
  for $f_{ij}\in L^2_\sigma$. Since $f$ is bounded, we have
  $f_{ij}=0$ unless $i=j=0$,
  and so $f$ is independent of the $z$-variable. The result
  now follows since the ODEs \eqref{eq:ODEpolar} with $\xi=0,
  \hat{\chi}_\lambda=0$  have no solutions with growth rates in the
  range $(2-m,0)$.  
\end{proof}

\subsection{The model space $X_0$}
We now consider the cone $X_0$, or rather $X_0\setminus
\mathbf{C}\times \{0\}$. In polar coordinates 
the corresponding metric is
\[ g_{X_0} = dr^2 + r^2 h_L, \]
where $h_L$ is the metric on the link $L$. The link $(L, h_L)$ is
incomplete (since we  remove the singular rays
from $X_0$), and its metric completion $\overline{L}$ is the double
suspension of the link of the cone $V_0$, since $X_0 =
\mathbf{C}\times V_0$. In particular $\overline{L}$ has a circle of
singularities modeled on the cone $V_0$. In the language of
\cite{DM14}, $\overline{L}$ is a smoothly
stratified space of depth 1, and $h_L$ is an iterated edge
metric. 

We define weighted H\"older norms on $(L, h_L)$ in terms of the
weight function $w$ given by smoothing out the distance from the singular
stratum. As above, the weighted norms can be expressed concisely in
terms of the conformally scaled metric $w^{-2}h_L$: 
\[ \Vert f \Vert_{C^{k,\alpha}_{\tau}} = \Vert w^{-\tau}
  f\Vert_{C^{k,\alpha}_{w^{-2}h_L}}. \]
The norm depends on the precise smoothing chosen, but we end up with
equivalent norms. The main result about the Laplacian on $L$ that we
need is the following.
\begin{prop}\label{prop:Linvert}
  Let $a\in \mathbf{C}$, and consider the map
  \[ \Delta_{h_L} + a : C^{k,\alpha}_{\tau}(L) \to C^{k-2,\alpha}_{\tau-2}(L), \]
  for $\tau \in (2-m, 0)$ (recall $m=2n-2$ is the real dimension of
  $V_0$). If $\mathrm{Im}\, a \ne 0$, or if $a\in \mathbf{R}$ avoids a
  discrete set of values, then $\Delta + a$ is invertible. 
\end{prop}
\begin{proof}
  This follows from the Fredholm theory for edge operators of
  Mazzeo~\cite{Maz91}. It implies that for our range of weights
  the image of $\Delta + a$ is $L^2$-orthogonal to its kernel, and
  moreover $\Delta$ has discrete real spectrum.  
\end{proof}

We now consider the cone $X_0$ over $L$. It is more convenient to
conformally scale the cone metric $g_{X_0}$ to a cylinder, using the
radial distance function $r$: 
\[ \tilde{g}_{X_0} = r^{-2} g_{X_0} = dt^2 + h_L, \]
where $t= \ln r$. The equation $\Delta_{g_{X_0}} f=u$ is then
equivalent to
\[ \Delta_{h_L}f + \partial_t^2 f + (2n-2) \partial_t f = e^{2t} u. \]
The relevant weighted H\"older norms, compatible with the definition
\eqref{eq:weightedHolder} can be formulated in terms of the further
conformal scaling $w^{-2}\tilde{g}_{X_0}$, using the weight function
$w$ on $L$ above. We then define
\[ \Vert f\Vert_{C^{k,\alpha}_{\delta, \tau}} = \Vert e^{-\delta t}
  w^{-\tau} f\Vert_{C^{k,\alpha}_{w^{-2}\tilde{g}_{X_0}}}. \]
Our goal is  to
show that the map 
\[ \mathcal{L} = \Delta_{h_L} + \partial_t^2 +(2n-2) \partial_t : C^{k,\alpha}_{\delta,
    \tau}(\mathbf{R}\times L) \to C^{k-2,\alpha}_{\delta,
    \tau-2}(\mathbf{R}\times L) \]
is invertible. It is convenient to conjugate the operator by
$e^{\delta t}$, to reduce to the case when $\delta = 0$. Writing
$\mathcal{L}_\delta(f) = e^{\delta t} \mathcal{L}(e^{-\delta t}f)$, we
have
\[ \mathcal{L}_\delta = \Delta_{h_L} + \partial_t^2 +
  (2n-2-2\delta)\partial_t + (\delta^2 - (2n-2)\delta), \]
and our main result is the following. 
\begin{prop}\label{prop:RLinvert}
  For $\delta$ avoiding a discrete set of indicial roots, and $\tau\in
  (2-m,0)$, the operator
  \[ \mathcal{L}_\delta : C^{k,\alpha}_\tau(\mathbf{R}\times L) \to
    C^{k-2,\alpha}_{\tau-2}(\mathbf{R}\times L) \]
  is invertible. Here we write $C^{k,\alpha}_{\tau} =
  C^{k,\alpha}_{0,\tau}$ for simplicity. 
\end{prop}

We first have the following result, analogous to
Proposition~\ref{prop:solve1}, except we do not need
the support to avoid $\{0\}\times L$.
\begin{prop}\label{prop:solve2}
  Suppose that $\delta$ avoids a discrete set of indicial roots. Let
  $\chi$ be smooth on $\mathbf{R}\times L$, such that $\hat{\chi}$ has
  compact support.  We can then find $f$ solving $\mathcal{L}_\delta f
  = \chi$, such that
  \begin{enumerate}
    \item $f\in C^{k,\alpha}_\tau$ for any $\tau < 0$, 
    \item For any $a > 0$ there is a constant $C$ such that $\Vert
      f(t, x)\Vert_{C^{k,\alpha}_\tau(|t| > A)}
      < C(1 + A)^{-a}$ for any $A > 0$. 
  \end{enumerate}
  The same applies to the operator
  \[ \mathcal{L}^*_\delta = \Delta_{h_L} + \partial_t^2 -
    (2n-2-2\delta)\partial_t + (\delta^2 -(2n-2)\delta), \]
  which is the adjoint of $\mathcal{L}_\delta$. 
\end{prop}
\begin{proof}
  The proof is similar to that of
  Proposition~\ref{prop:solve1}. Taking the Fourier transform in the
  $t$ variable we need
  to solve the equations
  \[ \label{eq:aaa5}
    \Delta_{h_L}\hat{f} - \big[\xi^2 \hat{f} - \sqrt{-1}\xi (2n-2-2\delta) -
    \delta^2 + (2n-2)\delta\big] \hat{f}  = \hat{\chi}. \]
  This equation is of the form $(\Delta_{h_L} + a)\hat{f} =
  \hat{\chi}$. Suppose that $\delta \ne n-1$.  If $\xi\ne 0$ then we have
  $\mathrm{Im}\,a \ne0$, while if $\xi = 0$, then $a = \delta^2
  -(2n-2)\delta$.  Therefore if $\delta$ is generic, we can apply 
  Proposition~\ref{prop:Linvert} to find $\hat{f}$ no matter what
  $\xi$ is. In addition, since $C^{k,\alpha}_\tau \subset C^{k,\alpha}_{\tau'}$ if $\tau >
  \tau'$, we have that $f\in C^0_{\tau'}$ for any $\tau' <
  0$. The Schauder estimates then imply that $f\in C^k_{\tau'}$ for
  any $k > 0$ as well, showing the property (1). To see (2) we claim
  that $\hat{f}$ is a smooth function of $\xi$. This can be seen by
  differentiating the equation as we did in
  Proposition~\ref{prop:solve1}. Inductively we find that each
  derivative $\partial^l_\xi \hat{f}$ satisfies an equation of the
  form
  \[ \Delta_{h_L} \partial^l_\xi \hat{f} + a(\xi) \partial^l_\xi
    \hat{f} = g(\xi), \]
  where $g(\xi)$ has compact support in $\xi$, and $g(\xi)\in
  C^{k,\alpha}_{\tau}$ for all $\tau < 0$. In addition $a(\xi)$ is
  such that Proposition~\ref{prop:Linvert} applies. In this way we can
  bound the $\xi$-derivatives of $\hat{f}$, and so the inverse Fourier
  transform $f$ will have the required decay. 
\end{proof}

This has the following corollary. 
\begin{cor}\label{cor:ker02}
  Suppose that $f\in C^{k,\alpha}_\tau$ satisfies $\mathcal{L}_\delta f
  = 0$ for some generic $\delta$, and $\tau\in (2-m, 0)$. Then $f=0$. 
\end{cor}
\begin{proof}
  Suppose that $f$ is nonzero.
  We can then find $\chi$ with $\hat{\chi}$
  having compact support such that $\int f\chi \ne 0$. We apply
  Proposition~\ref{prop:solve2} to solve $\mathcal{L}^*_\delta h =
  \chi$, with $h\in C^{k,\alpha}_{\tau'}$, where $\tau'$ is chosen
  so that $2-m-\tau < \tau' < 0$. We then have
  \[ \int_{\mathbf{R}\times L} f\chi \,dV &= \int_{\mathbf{R}\times L}
    f\mathcal{L}_\delta^*h \,dV \\
    &= \int_{\mathbf{R} \times L}
    (\mathcal{L}_\delta f)h\, dV = 0, \]
  where the integration by parts is justified by the decay properties
  of $h$ and the choice of $\tau'$. This contradicts our choice of $\chi$. 
\end{proof}

We next use a standard blowup argument to obtain the following. 
\begin{prop} \label{prop:closedrange1}
  Let $\tau\in (2-m, 0)$, and $\delta$ generic. There exists a constant $C$ such that 
  for any $f\in C^{k,\alpha}_\tau(\mathbf{R}\times L)$ we have
  \[ \Vert f\Vert_{C^{k,\alpha}_{\tau}} \leq C \Vert
    \mathcal{L}_\delta f\Vert_{C^{k-2,\alpha}_{\tau-2}}. \]
  In particular $\mathcal{L}_\delta$ has closed range. 
\end{prop}
\begin{proof}
  Note first that since the metric $w^{-2}\tilde{g}_{X_0}$ has bounded
  geometry, we can use the Schauder estimates to obtain 
  \[ \label{eq:Schauder1} \Vert f\Vert_{C^{k,\alpha}_\tau} \leq C(
    \Vert \mathcal{L}_\delta f
    \Vert_{C^{k-2,\alpha}_{\tau-2}} + \Vert f\Vert_{C^0_\tau}). \]
  If follows that it is enough to prove that 
  \[ \Vert f\Vert_{C^0_\tau} \leq C \Vert \mathcal{L}_\delta f
    \Vert_{C^{k-2,\alpha}_{\tau-2}} \]
  for a uniform constant $C$.

  We argue by contradiction. Suppose that we have a sequence of
  functions $f_i \in C^{k,\alpha}_\tau$ with $\Vert
  f_i\Vert_{C^0_\tau} = 1$, but $\Vert \mathcal{L}_\delta f_i
  \Vert_{C^{k-2,\alpha}_{\tau-2}} < 1/i$. We can find points $(t_i,
  x_i)\in \mathbf{R}\times L$ such that
  \[ |f_i(t_i, x_i)| > \frac{1}{2} w(x_i)^\tau. \]
  By translating in the $t$ direction we can assume that $t_i=0$ for
  each $i$. 

  There are two possibilities. If $w(x_i)$ is bounded away from zero,
  then by choosing a subsequence we can assume that $x_i \to x$ for
  some $x \in L$. The Schauder estimate \eqref{eq:Schauder1} implies
  that choosing a further subsequence we can assume that the $f_i$
  converge locally in $C^{k,\alpha'}$ to a limit $f\in C^{k,\alpha}_\tau$, which then must satisfy
  $\mathcal{L}_\delta f =0$. Corollary~\ref{cor:ker02} implies that
  $f=0$, which contradicts that $|f(x)| \geq \frac{1}{2}w(x)^\tau$. 

  The other possibility is that $w(x_i)\to 0$. Consider the rescaled
  metrics $w(x_i)^{-2} \tilde{g}_{X_0}$. If we take the pointed limit of $\mathbf{R}\times L$
  with these rescaled metrics, based at the points $(0,x_i)$, then (up
  to choosing a subsequence) we obtain the limit space $\mathbf{R}^2
  \times V_0$, with the product metric $\omega_{Euc} +
  \omega_{V_0}$, with basepoint $(0, x)$ for some $x\in V_0$ at
  distance 1 from the vertex. This is just the statement that $h_L$ is modeled on
  $S^1\times V_0$ near the circle of singularities. 

  Under taking this pointed limit the rescaled functions
  $w(x_i)^{-\tau} f_k$  converge locally in $C^{k,\alpha'}$ 
  to a limit $f \in C^{k,\alpha}_\tau(\mathbf{R}^2 \times V_0)$,
  satisfying $\Delta f =0$. Corollary~\ref{cor:ker0} implies that
  $f=0$, contradicting $|f(0,x)| \geq 1/2$. 
\end{proof}

We can finally prove Proposition~\ref{prop:RLinvert}. 
\begin{proof}[Proof of Proposition~\ref{prop:RLinvert}]
  It is enough to show that
  \[ \mathcal{L}_\delta : C^{k,\alpha}_\tau(\mathbf{R}\times L) \to
    C^{k-2,\alpha}_{\tau-2}(\mathbf{R}\times L) \]
  is surjective, since Corollary~\ref{cor:ker02} implies that it is
  injective. While we already know that the image is closed, and
  moreover Proposition~\ref{prop:solve2} provides us with many
  elements in the image, these functions do not form a dense set in
  $C^{k-2,\alpha}_{\tau-2}$, so we cannot immediately conclude. 

  Instead let $u \in C^{k-2,\alpha}_{\tau-2}(\mathbf{R}\times L)$. We can find a sequence
  of smooth $\chi_i$ such that $\hat{\chi}_i$ has compact support,
  $\chi_i \to u$ locally uniformly  and
  moreover
  \[ \Vert \chi_i\Vert_{C^{k-2,\alpha}_{\tau-2}} < C \] 
  for a constant $C$ depending on $u$. By
  Proposition~\ref{prop:solve2} we can find $f_i \in
  C^{k,\alpha}_\tau(\mathbf{R}\times L)$ such that $\mathcal{L}_\delta f_i = \chi_i$, and
  by Proposition~\ref{prop:closedrange1} we have 
  \[ \Vert f_i \Vert_{C^{k,\alpha}_\tau} < C, \]
  for $C$ independent of $i$. 
  Up to choosing a subsequence we can take a limit $f_i \to f$, with
  the convergence holding locally in $C^{k,\alpha'}$, and such that
  $\Vert f\Vert_{C^{k,\alpha}_\tau} \leq C$. The limit satisfies
  $\mathcal{L}_\delta f = u$, and so $\mathcal{L}_\delta$ is
  surjective. 
\end{proof}

\subsection{The model space $\mathbf{C}\times V_1$}
Let us now move on to the Laplacian on $\mathbf{C}\times V_1$. Here
the relevant weighted spaces are defined in terms of the weight
function $\zeta$ on $V_1$, which is a smoothed out version of $\max\{1,
d(\cdot, o)\}$ for a point $o\in V_1$. As above, the weighted
H\"older spaces can be defined in terms of a conformally scaled
version of the product metric $g = g_{Euc} + g_{V_1}$:
\[ \Vert f\Vert_{C^{k,\alpha}_\tau} = \Vert \zeta^{-\tau}
  f\Vert_{C^{k,\alpha}_{\zeta^{-2} g}}. \]
The weighted norms on $V_1$ are defined analogously. 

We first have the following result, analogous to
Proposition~\ref{prop:Linvert}.
\begin{prop}\label{prop:V1invert}
  Let $\tau\in (2-m,0)$, $\lambda \geq 0$, 
  and let $u$ be smooth with compact support on
  $V_1$. We can find a smooth function $f$ on $V_1$ such that
  \[ \label{eq:d1} \Delta_{V_1} f - \lambda f = u, \]
  and in addition we have an estimate $|f| \leq C \zeta^\tau$ for a
  constant that depends on $\Vert u\Vert_{C^{0}_{\tau-2}}$ and $\tau$, but not on
  $\lambda$. Recall that here $m=2n-2$ is the dimension of $V_1$. 
\end{prop}
\begin{proof}
  When $\lambda=0$, 
  this follows from the standard theory for the Laplacian acting in
  weighted spaces on the asymptotically conical manifold $V_1$ (see
  e.g. Lockhart-McOwen~\cite{LM85}):
  \[ \Delta_{V_1} : C^{k,\alpha}_{\tau}(V_1) \to
    C^{k-2,\alpha}_{\tau-2}(V_1). \]
  For our choice
  of weight the Laplacian is self-adjoint, and moreover any decaying element in
  $f \in \ker \Delta_{V_1}$ decays at the rate of at least $d(\cdot,
  o)^{2-m}$. It follows by integration by parts that $f$ is constant,
  but since it decays, we have in fact $f=0$. Hence $\Delta_{V_1}$ is
  invertible. 

  When $\lambda > 0$, then we can also solve Equation~\eqref{eq:d1}
  using that $\Delta_{V_1} - \lambda$ is an essentially self-adjoint
  operator on $L^2$, whose kernel is trivial. Moreover using the
  Schauder estimates it follows that the solution $f$ decays faster
  than any inverse power of $\zeta$. What remains is to obtain a uniform
  estimate for this decay, independent of $\lambda$ (in particular as
  $\lambda \to 0$). 

  For this, we first let $b\in C^{k,\alpha}_\tau$ be a solution of
  \[ \Delta_{V_1} b = - \zeta^{\tau-2}. \]
  By the maximum principle we have that $b > 0$. We claim that for
  sufficiently large $C$, depending on $\Vert u\Vert_{C^0_{\tau-2}}$,
  we have an estimate $|f| < Cb$, independent
  of $\lambda$. Let us show that $f < Cb$: if this estimate were to fail, then the
  function $f - Cb$ would achieve a maximum at some point $x_{max}$, using the fast
  decay of $f$. In particular $f(x_{max}) > 0$, and at the same time
  by the maximum principle
  \[ 0 &\geq \Delta(f - Cb)(x_{max}) \\
    &= u(x_{max}) + \lambda f(x_{max})
    + C\zeta^{\tau -2}(x_{max}) \\
    &> u(x_{max}) + C \zeta^{\tau-2}(x_{max}). \]
   If $C$ is chosen large depending on $\Vert u\Vert_{C^0_{\tau-2}}$,
   this is a contradiction. In a similar way one can prove that $f >
   -Cb$ for the same $C$. 
\end{proof}

The next result is analogous to Proposition~\ref{prop:solve2}.
\begin{prop}\label{prop:solve3}
  Suppose that $\chi$ is a smooth function on $\mathbf{C}\times V_1$
  such that $\hat{\chi}$ has compact support.  We can then solve the
  equation $\Delta f = \chi$ such that
  \begin{enumerate}
    \item $f\in C^{k,\alpha}_\tau$ for any $\tau > 2-m$, 
    \item If $\hat\chi$ is supported  away from $\{0\}\times
  V_1$, then in addition for any $a > 0$ there is a constant $C$ such that $\Vert
      f\Vert_{C^{k,\alpha}_\tau(|z| > A)}
      < C(1 + A)^{-a}$ for any $A > 0$. 
  \end{enumerate}
\end{prop}
\begin{proof}
  The proof is similar to the proofs of
  Propositions~\ref{prop:solve1},~\ref{prop:solve2}. After Fourier
  transforming, the relevant equations are
  \[ \label{eq:aaa6} \Delta_{V_1} \hat{f} - |\xi|^2 \hat{f} = \hat{\chi}.\] 
  Proposition~\ref{prop:V1invert} implies that we can solve these
  equations with uniform estimates on $\Vert
  \hat{f}\Vert_{C^0_\tau}$. Taking the inverse Fourier transform we
  obtain a solution of $\Delta f = \chi$, with $f\in
  C^0_\tau$. The Schauder estimates then imply that $f\in
  C^{k,\alpha}_\tau$.

  In order to get the decay property (2), we can argue as in
  Propositions~\ref{prop:solve1}, \ref{prop:solve2}, by
  differentiating Equation~\eqref{eq:aaa6}. For $\xi\ne0$, the
  solutions of \eqref{eq:aaa6} decay faster than any inverse power of $\zeta$ in the
  $V_1$-direction, and so inductively we find that each derivative
  $\partial^l_\xi \hat{f}$ has the same decay. Therefore $\hat{f}$ is
  a smooth function $\mathbf{C}\to C^0_\tau(V_1)$ with compact
  support,  and from this  we obtain the
  required decay for $f$ in the $\mathbf{C}$ direction. 
\end{proof}

We can next follow the proofs of Corollary~\ref{cor:ker0} and
Proposition~\ref{prop:closedrange1} closely to prove
\begin{prop}\label{prop:closedrange3}
  Let $\tau\in (2-m,0)$. There exists a constant $C$ such that for any
  $f\in C^{k,\alpha}_\tau(\mathbf{C}\times V_1)$ we have
  \[ \Vert f\Vert_{C^{k,\alpha}_\tau} \leq C \Vert
    \Delta f\Vert_{C^{k-2,\alpha}_{\tau-2}}. \]
\end{prop}

Finally we can prove the following. 
\begin{prop}\label{prop:CV1invert2}
  The Laplacian
  \[ \Delta_{\mathbf{C}\times V_1} : C^{k,\alpha}_{\tau}(\mathbf{C}\times V_1) \to
    C^{k-2,\alpha}_{\tau-2}(\mathbf{C}\times V_1) \]
  is invertible for $\tau\in (2-m, 0)$. 
\end{prop}
\begin{proof}
  The proof of this, based on Corollary~\ref{cor:ker0} and
 Propositions~\ref{prop:solve3}
  and~\ref{prop:closedrange3}
   is essentially identical to the proof of Proposition~\ref{prop:RLinvert}. 
\end{proof}

\section{Inverting the Laplacian}\label{sec:Laplaceinverse}
In this section we study the mapping properties of the Laplacian in
the weighted spaces that we have defined. The main result is the
following. 
\begin{prop}\label{prop:invert}
  Suppose that we choose $\tau\in (4-2n,0)$ and $\delta$ avoids a
  discrete set of indicial roots. For sufficiently large $A > 0$ the
  Laplacian
  \[ \Delta : C^{2,\alpha}_{\delta, \tau}(\rho^{-1}[A,\infty), \omega) \to
  C^{0,\alpha}_{\delta-2, \tau-2}(\rho^{-1}[A,\infty),\omega) \]
  is surjective with inverse bounded independently of $A$. 
\end{prop}
The proof of this result will take up the remainder of this section.
The strategy is to construct an approximate inverse for the
Laplacian by localizing the problem on the different regions of 
$\rho^{-1}[A, \infty)$ for sufficiently large $A$, studied in
Propositions~\ref{prop:R>L} and \ref{prop:R<L}, and then using the
inverses constructed on corresponding model spaces in 
Section~\ref{sec:modelspace}. 

Suppose that we have a function $u \in
C^{0,\alpha}_{\delta-2,\tau-2}(\rho^{-1}[A,\infty),\omega)$, with
norm $\Vert u\Vert_{C^{0,\alpha}_{\delta-2,\tau-2}} < 1$. 
The goal is to construct, once $A$ is sufficiently large, a
function $f = Pu$ on $X_1$ with $\Vert
f\Vert_{C^{2,\alpha}_{\delta,\tau}(X_1,\omega)} < C$ for a uniform $C$
such that
\[ \Vert \Delta f -
  u\Vert_{C^{0,\alpha}_{\delta-2,\tau-2}(\rho^{-1}[A,\infty),\omega)}
  < \frac{1}{2}. \]
 Then the operator $P$ is an approximate inverse for $\Delta$, so
 $\Delta P$ is invertible, and as a consequence 
 $\Delta$ has a bounded right inverse.

 Using the extension map defined in Proposition~\ref{prop:extensionop}
 we can assume that $u$ is actually defined on all of
$X_1$ and
\[ \Vert u\Vert_{C^{0,\alpha}_{\delta-2,\tau-2}(X_1,\omega)} < C, \]
for a constant $C$ independent of $A$. 
We will decompose $u$ using cutoff functions into various different
pieces. Let us recall the cutoff functions $\gamma_1,\gamma_2$ from
before, so that $\gamma_1(s)$ is supported where $s > 1$, and
$\gamma_1 + \gamma_2=1$. Let us choose a large number $\Lambda$, and
write $u= u_1 + u_2$, where
\[ u_i = \gamma_i(R\Lambda^{-1}\rho^{-1/d}) u. \]
So $u_1$ is supported on the set where $R > \Lambda
\rho^{1/d}$. Let us use the notation $\mathcal{U} = \{R >
\Lambda\rho^{1/d}\}\cap \{\rho > A\}$ from Proposition~\ref{prop:R>L}.
We then have a map $G:\mathcal{U} \to X_0$ such that
\[ \label{eq:m1} \Vert G^*\omega_{X_0} - \omega\Vert_{C^{k,\alpha}_{0,0}} <
  \epsilon. \]
We can use the map $G$ to view $u_1$  (at least on the set where $\rho
> A$) as a function on $X_0$, supported away from $\mathbf{C}\times
0$. 
Proposition~\ref{prop:RLinvert} implies 
\begin{prop}\label{prop:model1}
  As long as $\tau\in (4-2n, 0)$, and $\delta$ avoids a discrete set
  of indicial roots, the Laplacian
  \[ \Delta_{X_0} : C^{k,\alpha}_{\delta,\tau}(X_0) \to
    C^{k-2,\alpha}_{\delta-2,\tau-2}(X_0) \]
  is invertible. 
\end{prop}

We can therefore define $Pu_1$ on $X_0$ satisfying $\Delta_{X_0}
Pu_1 = u_1$, and satisfying the estimate
\[ \Vert Pu_1\Vert_{C^{2,\alpha}_{\delta,\tau}(X_0)} < C, \]
for a uniform $C$. In order to transfer this function back to $X_1$,
we use another cutoff function
\[ \beta_1 = \gamma_1\left(\frac{\ln (R\Lambda^{-1/2}\rho^{-1/d})}{\ln
      \Lambda^{1/4}}\right). \]
This has the property that $\beta_1 = 0$ on the region where $R <
\Lambda^{3/4} \rho^{1/d}$, while $\beta_1=1$ on the support of
$\gamma_1(R\Lambda^{-1}\rho^{-1/d})$, which is where $u_1$ is
supported. Moreover in our weighted spaces
we have the estimate
\[ \label{eq:a5} \Vert \nabla \beta_1
  \Vert_{C^{k,\alpha}_{-1,-1}(\rho^{-1}(1,\infty)\cap X_0)} <
  \frac{C}{\ln \Lambda}. \]
We have
\[ \Delta_{X_0}(\beta_1Pu_1) = u_1 +
2  \nabla\beta_1\cdot \nabla(Pu_1) + (\Delta_{X_0}\beta_1) Pu_1, \]
and so using the multiplication properties
\[ \Vert fg\Vert_{C^{k,\alpha}_{a+b, c+d}} \leq C \Vert
  f\Vert_{C^{k,\alpha}_{a,c}} \Vert g\Vert_{C^{k,\alpha}_{b,d}} \]
together with \eqref{eq:a5} we obtain
\[ \label{eq:aa1} \Vert \Delta_{X_0}(\beta_1 Pu_1) - u_1
  \Vert_{C^{0,\alpha}_{\delta-2,\tau-2}(\rho^{-1}(1,\infty)\cap
    X_0)} < \frac{C}{\ln \Lambda}. \]
Using the map $G$ again to view $\beta_1Pu_1$ as a function on $X_1$, 
and using \eqref{eq:m1} to compare the Laplacians on
$X_0$ and $X_1$ we find that
\[ \label{eq:aa2} \Vert \Delta_\omega(\beta_1 Pu_1) - u_1
  \Vert_{C^{0,\alpha}_{\delta-2,\tau-2}(\rho^{-1}[A,\infty))} <
  \epsilon, \]
once $\Lambda$ and $A$ are sufficiently large. 

We next need to examine the piece $u_2$, which is supported in the
region where $R < 2\Lambda \rho^{1/d}$, and we are assuming that 
in addition $\rho > A$. Note that here we have $\rho \sim
|z|$. Geometrically this region can be thought of roughly as
a fibration over the set $\{ |z| > A \} \subset \mathbf{C}$, whose
fiber over $z$ is the region $\{ r < 2\Lambda \} \subset V_1$, scaled
down by a factor of $|z|^{1/d}$. We decompose $u_2$ into pieces whose
supports are localized in the $z$-plane. Proposition~\ref{prop:R<L}
tells us that suitably
scaled, on these regions we can approximate our space with a
corresponding region in the product $\mathbf{C} \times V_1$. By
Proposition~\ref{prop:CV1invert2}, we can invert
the Laplacian there. 

Let us choose a large $B > 0$. We construct cutoff functions $\chi_i$
on $\mathbf{C}$ such that $\sum\chi_i = 1$ on the set where
$|z| > A$ as follows. Consider
$(\mathbf{C}, \tilde{g})$, where $\tilde{g} = B^{-2}|z|^{-2/d} g_{Euc}$ is a
conformal scaling of the Euclidean metric. Since $d > 1$, 
on the set where $|z| > A$ for sufficiently large $A$, this metric is
close to being Euclidean on larger and larger scales. 
We can then cover this region with disks of
radius 2 (in the metric $\tilde{g}$) centered at points $z_i$,
such that the corresponding disks
of radius 1 are disjoint, and define the cutoff funtions $\chi_i$
supported in the disks of radius 2, and equal to 1 on the disks of
radius 1. Scaling back the metric on $\mathbf{C}$ we have that
$\chi_i=1$ on the ball of radius $B|z_i|^{1/d}$ around
$z_i$, and $\chi_i=0$ outside of the ball of radius $2B|z_i|^{1/d}$
around $z_i$. In addition $|\nabla^l \chi_i|_{g_{Euc}} =
O(B^{-l}|z_i|^{-l/d})$ for all $l \geq 0$. 

We define another set of cutoff functions
$\tilde{\chi}_i$ in a similar way, which we will use to transfer our local
solutions back to $X_1$ (analogous to $\beta_1$ above). The
$\tilde{\chi_i}$ equal 1 on the support of $\chi_i$, and are supported
in the balls of radius $3B|z_i|^{1/d}$. Furthermore $|\nabla^l
\tilde{\chi}_i| = O(B^{-l}|z_i|^{-l/d})$ as well. An additional
important property of these cutoff
functions, which can be seen more clearly in terms of the conformally
scaled metric $\tilde{g}$, is that any $z$ with $|z| > A$ is in the support of only a
fixed bounded number $N$ of the $\tilde{\chi}_i$ ($N$ is independent of the choices of
large $B, A$). 

We now decompose $u_2$ into the pieces $u_2 = \sum_i \chi_i u_2$, at
least on the region where $|z| > A$. By construction the function
$\chi_i u_2$ is supported on a region where $z \in B(z_i, 2B
|z_i|^{1/d})$ for a point $z_i \in \mathbf{C}$, and in addition $R <
2\Lambda |z|^{1/d}$. By Proposition~\ref{prop:R<L}, on this region the
scaled metric $|z_i|^{-2/d}\omega$ can be approximated by the product
metric $\omega_{\mathbf{C}}\times \omega_{V_1}$ on $\mathbf{C}\times
V_1$, using the map $H$. Here we have the following result from
Section~\ref{sec:modelspace}. 
\begin{prop}\label{prop:CV1invert}
  The Laplacian
  \[ \Delta_{\mathbf{C}\times V_1} : C^{k,\alpha}_{\tau}(\mathbf{C}\times V_1) \to
    C^{k-2,\alpha}_{\tau-2}(\mathbf{C}\times V_1) \]
  is invertible for $\tau\in (4-2n, 0)$. 
\end{prop}

In order to apply this, we need to view $\chi_i u_2$ as a function on
$\mathbf{C}\times V_1$, using the map $H$, and relate its weighted
norm in $C^{0,\alpha}_{\tau-2}$ on $\mathbf{C}\times V_1$ to the
norm in $C^{0,\alpha}_{\delta-2,\tau-2}$ on $X_1$. For this, note
that on our region we have $\rho \sim |z_i|$, and so by the definition
of the weight function $w$ we have
\[ \rho^{-2}w^{-2}&\sim \max\{|z_i|^{1/d},R\}^{-2} \\
\rho^{\delta-2}w^{\tau-2} &\sim |z_i|^{\delta-\tau}
\max\{|z_i|^{1/d},R\}^{\tau-2}. \]
At the same time the weight function $\hat{\zeta}$ used to define the
weighted spaces on $\mathbf{C} \times V_1$ (using the
notation from Proposition~\ref{prop:R<L}), is comparable to $\max\{1,
|z_i|^{-1/d}R\}$. It follows that the estimate $\Vert
u_2\Vert_{C^{0,\alpha}_{\delta-2, \tau-2}} < C$ on $X_1$ translates to 
\[ \label{eq:uu1} \Vert u_2\Vert_{C^{0,\alpha}_{\tau-2}(\mathbf{C}\times V_1)} < C
|z_i|^{\delta-\tau + \frac{\tau-2}{d}} \]
on the support of $\chi_i$. Note also that by construction, on the
support of $u_2$ we have  
\[ \label{eq:uu2}|\nabla^l \chi_i| < CB^{-l} < CB^{-l}\Lambda^l \hat{\zeta}^{-l}, \]
 thinking of $\chi_i$ as a function on $\mathbf{C}\times V_1$ and
 using the Euclidean metric on $\mathbf{C}$ (since
 $\hat{\zeta} < C\Lambda$). It follows that once $B$ is sufficiently
 large, depending on $\Lambda$, the estimate \eqref{eq:uu1} implies
\[ \Vert \chi_i u_2\Vert_{C^{0,\alpha}_{\tau-2}(\mathbf{C}\times V_1)} < C
|z_i|^{\delta-\tau + \frac{\tau-2}{d}}. \]
We now apply Proposition~\ref{prop:CV1invert}, but note that 
$\Delta_{|z_i|^{-2/d}\omega} = |z_i|^{2/d}
\Delta_\omega$. We therefore use the Proposition to define $P(\chi_i
u_2)$ by
\[ \Delta_{\mathbf{C}\times V_1} P(\chi_i u_2) = |z_i|^{2/d} \chi_i
  u_2, \]
satisfying the bound
\[ \Vert P(\chi_i u_2)\Vert_{C^{2,\alpha}_\tau} < C |z_i|^{\delta-\tau
  + \frac{\tau}{d}}. \]
We need to transfer this function back to the manifold $X_1$. Note
that in terms of the coordinate $\hat{z}_i$ from
Proposition~\ref{prop:R<L}
on the $\mathbf{C}$ factor and the
distance function $\hat{\zeta}$ on $V_1$, the function
$\chi_i u_2$ is supported in the region where $|\hat{z}_i| < 2B$ and $\hat{\zeta} < 4
\Lambda$. We use the cutoff functions $\tilde{\chi}_i$ from above, 
which equal 1 on the supports of $\chi_i$, and are supported, 
in these coordinates, where $|\hat{z}_i| < 3B$. In addition we use the cutoff
function
\[ \beta_2 = \gamma_2\left( \frac{\ln (\hat{\zeta} 4^{-1} \Lambda^{-1})}{\ln
      \Lambda}\right), \]
which equals 1 where $\hat{\zeta} < 4\Lambda$, vanishes where
$\hat{\zeta} >
4\Lambda^2$ and has the property that in our weighted
spaces
\[ \label{eq:a3} \Vert \nabla \beta_2\Vert_{C^{k,\alpha}_{-1}} < \frac{C}{\ln
    \Lambda}. \]
We now need an estimate analogous to \eqref{eq:aa1} for the difference
\[ \Delta_{\mathbf{C}\times V_1}(\tilde{\chi}_i\beta_2 P(\chi_i u_2)) -
|z_i|^{2/d} \chi_i u_2 =2 \nabla (\beta_2\tilde{\chi}_i)\cdot \nabla P(\chi_i u_2) +
\Delta_{\mathbf{C}\times V_1}(\beta_2\tilde{\chi}_i) P(\chi_i u_2).\] 
Using the estimate \eqref{eq:a3} for $\beta_2$ and an estimate
analogous to \eqref{eq:uu2} for $\tilde{\chi}_i$ (with $\Lambda$
replaced by $\Lambda^2$ and $B$ chosen correspondingly larger), we find that
\[ \Vert \Delta_{\mathbf{C}\times V_1} (\tilde{\chi}_i\beta_2 P(\chi_i u_2)) -
|z_i|^{2/d} \chi_i u_2 \Vert_{C^{0,\alpha}_{\tau-2}(\mathbf{C}\times
  V_1)} < C\epsilon |z_i|^{\delta-\tau + \frac{\tau}{d}}. \]
Proposition~\ref{prop:R<L} allows us to estimate the difference
between $|z_0|^{-2/d}\omega$ and the product metric on $\mathbf{C}\times
V_1$ under the identification using the map $H$, and this leads to
\[  \Vert \Delta_{|z_i|^{-2/d}\omega} (\tilde{\chi}_i\beta_2 P(\chi_i u_2)) -
|z_i|^{2/d} \chi_i u_2 \Vert_{C^{0,\alpha}_{\tau-2}(\mathbf{C}\times
  V_1)} < C\epsilon |z_i|^{\delta-\tau + \frac{\tau}{d}}. \]
Dividing through by $|z_i|^{2/d}$, and translating the estimate back
to our weighted spaces on $X_1$, we get
\[ \label{eq:uu3} \Vert \Delta_\omega(\tilde{\chi}_i\beta_2 P(\chi_i u_2)) - \chi_i
u_2 \Vert_{C^{0,\alpha}_{\delta-2,\tau-2}} < C\epsilon. \]

Finally we define
\[ Pu = \beta_1 Pu_1 + \sum_i \beta_2 \tilde{\chi}_i P(\chi_i u_2). \]
We use the estimates \eqref{eq:aa2} and \eqref{eq:uu3}, together with
the observation that any given point is contained in at most a fixed
number of our regions, to deduce that
\[ \Vert \Delta(Pu) - u\Vert_{C^{0,\alpha}_{\delta-2,\tau-2}} <
  C\epsilon < 1/2, \]
for sufficiently small $\epsilon$, which by the above discussion we
can achieve by first choosing $\Lambda$, then $B$, and finally $A$
sufficiently large. In this case $\Delta P$ is invertible, with
$\Vert (\Delta P)^{-1}\Vert < 2$. Since by
construction the operator $P$ has bounded norm independent of $A$, we
have constructed a right inverse $P(\Delta P)^{-1}$ for $\Delta$ for
sufficiently large $A$, with norm independent of $A$. This completes
the proof of Proposition~\ref{prop:invert}. 

\section{Calabi-Yau metrics on $\mathbf{C^n}$}\label{sec:perturb}
In Section~\ref{sec:approx1} we wrote down a form $\omega$ on the
hypersurface $X_1$, whose Ricci
potential decays in a suitable weighted space by
Proposition~\ref{prop:decayest2}. 
Our goal is to use the linear theory developed in
Section~\ref{sec:Laplaceinverse} to improve
the decay of the Ricci potential enough to be able
to apply Hein~\cite[Proposition 4.1]{HeinThesis} to construct a global
Calabi-Yau metric on $X_1$. Since we will use a similar method in
Section~\ref{sec:singularity} below, we will actually perturb $\omega$
to a metric $\tilde{\omega}$ which is Calabi-Yau on the set $\rho^{-1}[A,\infty)$ for
sufficiently large $A$.

\begin{prop}\label{prop:perturbed}
  Suppose that $A$ is sufficiently large, $\tau <0$ is sufficiently
  close to 0, and $\delta <2/d$ is as in Proposition~\ref{prop:decayest2}. 
  Then there exists a small
  $u\in C^{2,\alpha}_{\delta, \tau}$ such that 
  \[ (\omega + \ddb u)^n = (\sqrt{-1})^{n^2} \Omega\wedge \overline{\Omega} \]
  on the set $\rho^{-1}[A,\infty)$. 
\end{prop}
\begin{proof}
  Note first that if $\Vert u\Vert_{C^{2,\alpha}_{\delta,\tau}} <
  \epsilon$, then by the definition of the weighted norms we have
  \[ |\nabla^2 u|_\omega < C\epsilon \rho^{\delta-2} w^{\tau-2}. \]
  Since $w > C^{-1} \rho^{1/d-1}$ this implies 
  \[ |\nabla^2 u|_\omega < C\epsilon \rho^{\delta - 2/d -
    \tau(1-1/d)}. \]
  If $\tau$ is close to zero, and $\delta < 2/d$, then for
  sufficiently small $\epsilon$, the form $\omega+\ddb u$ defines a
  metric uniformly equivalent to $\omega$. This is the reason for our
  requirement that $\delta < 2/d$ in
  Proposition~\ref{prop:decayest2}. 

  Let us define
  \[ \mathcal{B} = \{ u\in C^{2,\alpha}_{\delta, \tau}\,:\, \Vert
  u\Vert_{C^{2,\alpha}_{\delta, \tau}} \leq \epsilon_0\}, \]
  where $\epsilon_0$ is sufficiently small so that $\omega + \ddb u$
  is uniformly equivalent to $\omega$. 
  Let us define the operator
  \[ F :  \mathcal{B} &\to C^{0,\alpha}_{\delta-2,
    \tau-2}(\rho^{-1}[A,\infty))\\
  u&\mapsto \left. \log\frac{(\omega+ \ddb u)^n}{(\sqrt{-1})^{n^2}\Omega\wedge
      \overline{\Omega}}\right|_{\rho^{-1}[A, \infty)}.\]
  Our goal is to find $u\in\mathcal{B}$ such that $F(u)=0$. 

  Let us write
  \[ \label{eq:Qdef} F(u) = F(0) + \Delta_\omega u + Q(u) \]
  for a suitable nonlinear operator $Q$. Denoting by $P$ the right inverse
  for $\Delta$ found in Proposition~\ref{prop:invert}, it is enough to
  solve
  \[ u = P( -F(0) - Q(u)), \]
  i.e. we are looking for a fixed point of the map $N(u) =
  P(-F(0)-Q(u))$. Note that we have a uniform bound for $P$
  independent of $A$ (for sufficiently large $A$), since the right
  inverse for one choice of $A$ also provides a right inverse for all
  larger choices of $A$. In addition either from an explicit formula
  for $Q$, or from differentiating Equation~\eqref{eq:Qdef} and
  estimating the difference $\Delta_{\omega+\ddb u} -
  \Delta_{\omega+\ddb v}$, 
  we see that as long as $u, v\in \mathcal{B}$  we have the estimate
  \[ \Vert Q(u) - Q(v)\Vert_{C^{0,\alpha}_{\delta-2, \tau-2}} \leq C
  (\Vert u\Vert_{C^{2,\alpha}_{2,2}} + \Vert
  v\Vert_{C^{2,\alpha}_{2,2}}) \Vert u -
  v\Vert_{C^{2,\alpha}_{\delta,\tau}}. \]
  For this note that the $C^{2,\alpha}_{2,2}$ norms of $u$ controls 
  $\ddb u$ in $C^{0,\alpha}_{0,0}$, which in turn is the
  difference between the metrics $\omega$ and $\omega + \ddb u$. 

  It follows that there is a constant $\epsilon_1$ such that 
  $N$ is a contraction on $\mathcal{B}$ as long as 
  \[ \Vert u\Vert_{C^{2,\alpha}_{2,2}} < \epsilon_1 \]
  for all $u\in \mathcal{B}$. This holds if $\epsilon_0$ in the
  definition of $\mathcal{B}$ is sufficiently small, since as above we have
  \[ \rho^{\delta} w^\tau \leq C \rho^{\delta-2 + (\tau-2)(1/d-1)}
  \rho^2 w^2, \]
  which implies 
  \[ \Vert u\Vert_{C^{2,\alpha}_{2,2}} \leq C \Vert
  u\Vert_{C^{2,\alpha}_{\delta,\tau}}. \]
  We can therefore assume that $\Vert N(u) - N(v)\Vert < \frac{1}{2}
  \Vert u-v\Vert$ for $u,v\in \mathcal{B}$. 
  
  Finally we just have to ensure that $N$ maps $\mathcal{B}$ to
  itself. For this first note that by the estimates of
  Proposition~\ref{prop:decayest2} we have $F(0) \in
  C^{0,\alpha}_{\delta'-2, \tau-2}$ for some $\delta' < \delta$ that is
  sufficiently close to $\delta$. It follows that
  \[ \Vert F(0)\Vert_{C^{0,\alpha}_{\delta-2,
      \tau-2}(\rho^{-1}[A,\infty))} < CA^{\delta'-\delta}, \]
  which can be made arbitrarily small by choosing $A$ large. 
  Next,  we have that if $u\in \mathcal{B}$, then 
  \[ \Vert N(u) \Vert_{C^{2,\alpha}_{\delta, \tau}} &\leq \Vert N(0)
  \Vert_{C^{2,\alpha}_{\delta, \tau}} + \Vert N(u) - N(0)
  \Vert_{C^{2,\alpha}_{\delta,\tau}} \\
  &\leq C \Vert
  F(0)\Vert_{C^{0,\alpha}_{\delta-2,\tau-2}(\rho^{-1}[A,\infty))} +
  \frac{1}{2}\Vert u\Vert_{C^{2,\alpha}_{\delta, \tau}} \\
  &\leq CA^{\delta'-\delta} +  \frac{\epsilon_0}{2}. \]
  For sufficiently large $A$ we therefore have $N(u)\in
  \mathcal{B}$. Therefore we can find a fixed point of $N$, as required.
\end{proof}

We can now complete the proof of Theorem~\ref{thm:1}, by applying
Hein~\cite[Proposition 4.1]{HeinThesis} to the metric $\tilde{\omega}
= \omega + \ddb u$ constructed in the previous proposition. Note that
by elliptic regularity $\tilde{\omega}$ is actually
smooth. Since $\tilde{\omega}$ is asymptotically a small perturbation
of $\omega$, Proposition~\ref{prop:tancone} shows that the tangent
cone at infinity of $(X_1,\tilde{\omega})$ is the cone $X_0$.  
To apply Hein's result we need to check that $(X_1, \tilde{\omega})$
satisfies the condition $SOB(2n)$ (see \cite[Definition
3.1]{HeinThesis}), and it has a $C^{3,\alpha}$
quasi-atlas. The condition $SOB(2n)$ includes the connectedness of
certain annuli on $X_1$, but for the proof it is actually 
enough to check the RCA condition
used in Degeratu-Mazzeo~\cite{DM14} for instance. Since
$\tilde{\omega}$ is uniformly equivalent to $\omega$, this condition
holds by Proposition~\ref{prop:RCA}. To check that
$(X_1,\tilde{\omega})$ satisfies $SOB(2n)$ it is then enough to show
that for a constant $C>0$ and any $s > C$ the volume of the ball
$B(x,s)$ satisfies
\[ C^{-1} s^{2n} < \mathrm{Vol}(B(x,s)) < Cs^{2n}, \]
for any $x\in X_1$. This can be seen for instance by using Colding's
volume convergence theorem~\cite{Colding} under Gromov-Hausdorff limits,
and by noting that $\tilde{\omega}$ is Ricci flat outside of a compact
set, and has a tangent cone at inifinity that is non-collapsed
(i.e. has Euclidean volume growth). 

The existence of a $C^{3,\alpha}$ quasi-atlas, i.e. charts
of a uniform size around each point in which the metric is controlled
in $C^{3,\alpha}$, can also be seen using Propositions~\ref{prop:R>L},
\ref{prop:R<L}. These results show that for the metric $\omega$ we
actually have charts of size $\rho w$ around any point, in which the
metric is controlled in $C^{k,\alpha}$. Note that since $w >
\kappa^{-2}\rho^{1/d-1}$ we have $\rho w > \kappa^{-2}\rho^{1/d}$,
which goes to infinity as $\rho\to\infty$. When we perturb the metric
to $\tilde{\omega}$, then by our construction we a priori only control
$\tilde{\omega}$ in $C^{2,\alpha}$ in these charts, but elliptic
regularity allows us to improve this to $C^{k,\alpha}$. 

We can therefore apply Proposition 4.1 from Hein~\cite{HeinThesis} to
further perturb $\tilde{\omega}$ to a global Calabi-Yau metric on
$X_1$. This perturbation does not change the tangent cone at infinity,
and so we obtain the required Calabi-Yau metric on $\mathbf{C}^n \cong
X_1$ with tangent cone $X_0 = \mathbf{C}\times V_0$ at infinity.

\section{Calabi-Yau metrics in a neighborhood of an isolated
  singularity} \label{sec:singularity}
In this section we show that the ideas developed in the previous
sections can also be used to construct Calabi-Yau metrics on
neighborhoods of certain isolated singularities $(X_1, 0)$, whose
tangent cone at the singularity has singular cross section, as described in the
introduction. This is a generalization of unpublished work of
Hein-Naber~\cite{HNpreprint}, who considered the case when $X_1$ is the hypersurface
\[ \label{eq:HNeq} z^p + x_1^2 + \ldots + x_n^2 = 0, \]
with $p > 2\frac{n-1}{n-2}$. 

We consider the situation where $X_1 \subset
\mathbf{C}^{n+1}$ is the hypersurface
\[ z^p + f(x_1,\ldots, x_n) = 0, \]
where as before, $V_0 = f^{-1}(0) \subset \mathbf{C}^n$ admits a Calabi-Yau
cone metric. As before, we let the degree of $f$ be $d$
for the homothetic action with
 weight $w=(w_1,\ldots, w_n)$ on $x$, and this time we
require that $p > d$. For the hypersurface \eqref{eq:HNeq} the
condition $p > d$ coincides with Hein-Naber's condition $p >
2\frac{n-1}{n-2}$. Our goal is to construct a Calabi-Yau metric in
a neighborhood of the singular point $0\in X_1$, whose tangent cone at
0 is $\mathbf{C}\times V_0$. 

As before, we have the nowhere vanishing holomorphic $n$-form
\[ \Omega = \frac{dz\wedge dx_2 \wedge\ldots \wedge
    dx_n}{ \partial_{x_1}f} \]
on $X_1$, and the first step is to write down a metric $\omega$ on
$X_1$ whose Ricci potential
\[ h = \log\frac{\omega^n}{(\sqrt{-1})^{n^2}\Omega\wedge \overline{\Omega}} \]
decays in a suitable weighted space near the origin. The definition
of $\omega$ is completely analogous to \eqref{eq:approx1}, given by
\[ \omega = \ddb\Big(|z|^2 + \gamma_1(R\rho^{-\alpha})r^2 +
  \gamma_2(R\rho^{-\alpha}) |z|^{2p/d} \phi(z^{-p/d}\cdot x)\Big), \]
where this time we choose $\alpha \in (1, p/d)$, and $\gamma_i,\rho, R$ are
just as before. Note that the potential is asymptotic to $|z|^2 +
r^2$ as $\rho\to 0$

In analogy with Proposition~\ref{prop:decayest2} we have the
following. 
\begin{prop}\label{prop:decayest3}
  The form $\omega$ defines a metric on $X_1\setminus \{0\}$ on
  the set where $\rho < P^{-1}$
  for sufficiently large $P$. In addition we can find a
  weight $\delta > 2p/d$ for which the Ricci potential $h$ of $\omega$ satisfies
  \[ |\nabla^i h|_{\omega} < \begin{cases} 
      C_i\rho^{\delta-2-i}\, \text{ if } R > \kappa\rho, \\
      C_i\rho^\delta R^{-2-i}\, \text{ if } R\in
      (\kappa^{-1}\rho^{p/d}, \kappa\rho),\\
      C_i \rho^{\delta-2p/d - ip/d}\, \text{ if } R <
        \kappa^{-1}\rho^{p/d}, \end{cases}
    \]
   for suitable $\kappa, C_i > 0$. Recall that by Lemma~\ref{lem:d>2} we can
   assume $d > 2$. 
\end{prop}
\begin{proof}
  The proof of these estimates is very similar to the proof of
  Proposition~\ref{prop:decayest2}, estimating the Ricci potential by
  comparing $\omega$ to various model metrics in different regions, by
  scaling. We will keep the notation the same as in the proof of
  Proposition~\ref{prop:decayest2} to make the similarities
  apparent. The main difference is that now we are interested in
  estimating the errors as $\rho \to 0$ rather than $\rho \to
  \infty$. 
  
  \bigskip
  {\bf Region I}: Consider the region where $R > \kappa\rho$, and
  $\rho\in (D/2, 2D)$ for small $D$. We scale the metric to
  $D^{-2}\omega$ and use coordinates
  \[ \tilde{z} &= D^{-1}z, \quad \tilde{x} = D^{-1}\cdot x, \quad
    \tilde{r} = D^{-1}r. \]
  We have
  \[ D^{-2}\omega = \ddb( |\tilde{z}|^2 + \tilde{r}^2), \]
  and the equation of $X_1$ is
  \[ D^{p-d} \tilde{z}^p + f(\tilde{x}) = 0. \]
  As before, we obtain 
  \[ |\nabla^i h|_{D^{-2}\omega} \leq C_i D^{p-d}. \]
  If $p > d > 2$, then $p-d > 2p/d - 2$, and so we can choose $\delta
  > 2p/d$ satisfying $D^{p-d} < D^{\delta - 2}$. 

  \bigskip
  {\bf Region II}: Here $R\in (K/2, 2K)$ for some $K\in (4\rho^\alpha,
  \kappa\rho)$, and $\rho\in (D/2, 2D)$. Let $z$ be close to $z_0$
  such that $|z_0|\sim D$. In terms of coordinates
  \[ \tilde{z} = K^{-1}(z-z_0), \quad \tilde{x} = K^{-1}\cdot x, \quad
    \tilde{r} = K^{-1} r, \]
  we have
  \[ K^{-2}\omega = \ddb(|\tilde{z}|^2 + \tilde{r}^2), \]
  and the equation of $X_1$ is
  \[ K^{-d}(K\tilde{z}+z_0)^p + f(\tilde{x}) = 0, \]
  where $|\tilde{z}| \leq 1$. 
  Arguing as before, when we compare this to the equation
  $f(\tilde{x})=0$ we obtain an error of order $K^{-d}D^p$. This implies
  \[ |\nabla^i h|_{K^{-2}\omega} \leq C_i K^{-d}D^p = C_i K^{2-d}D^p K^{-2}. \]
  On this region $K > 4\rho^\alpha$ (and $d > 2$), so we have
  \[ K^{2-d}D^p < D^{p + (2-d)\alpha}. \]
  Since $\alpha < p/d$ we have $p + (2-d)\alpha > 2p/d$, so we can
  choose $\delta > 2p/d$. 

  \bigskip
  {\bf Region III}: Here $R\in (K/2, 2K)$ for $K\in (\rho^\alpha,
  2\rho^\alpha)$, and $\rho\in (D/2, 2D)$. We consider $z$ close to
  $z_0$, and since in this region $\rho$ is comparable to $|z|$, we
  have $|z_0| \sim D$. With the same scaling as in Region II, we have
  \[ K^{-2}\omega = \ddb\Big( |\tilde{z}|^2 + \gamma_1 \tilde{r}^2 +
    \gamma_2 K^{-2}|K\tilde{z}+z_0|^{2p/d}\phi\big((K\tilde{z}+z_0)^{-p/d}K\cdot
    \tilde{x}\big)\Big). \]
  The equation of $X_1$ is
  \[ K^{-d}(K\tilde{z}+z_0)^p + f(\tilde{x}) =0.\]
  As in Proposition~\ref{prop:decayest2} we compare this to the metric 
  \[ \ddb(|\tilde{z}|^2 + \tilde{r}^2) \]
  on $X_0$ with equation $f(\tilde{x})=0$. To compare the potentials,
  similarly to \eqref{eq:vv1} we have
  \[ \nabla^i\Big[ K^{-2}|K\tilde{z} + z_0|^{2p/d}
     \phi((K\tilde{z}+z_0)^{-p/d}K\cdot \tilde x) - \tilde{r}^2 \Big]=
     O\left( (K^{-1}D^{p/d})^{2+c}\right). \]

   Arguing as before we have an error of order $K^{-d}D^p$ from
   comparing the two  equations, which is bounded in the same way as
   in Region II.  Since in this region $K\sim D^\alpha$,
   the new error from comparing the K\"ahler potentials
   satisfies
   \[ (K^{-1}D^{p/d})^{2+c} = K^{-c}D^{\frac{p}{d}(2+c)}K^{-2} =
     D^{2\frac{p}{d} + c\left(\frac{p}{d}-\alpha\right)}K^{-2}. \]
   We need $\delta$ so that this is bounded by $D^\delta K^{-2}$ (as
   $D\to 0$).  Since $\alpha < p/d$, we can choose $\delta > 2p/d$. 
   \bigskip
   
  {\bf Region IV}: Here $R\in (K/2, 2K)$ with $K\in
  (\kappa^{-1}\rho^{p/d}, \rho^\alpha/2)$, and $\rho\in (D/2,2D)$. We
  choose $z$ close to $z_0$, with $|z_0|\sim D$. We scale the same way
  as in Regions II, III. As before, we have
  \[ K^{-2}\omega = \ddb\Big( |\tilde{z}|^2 +
    K^{-2}|K\tilde{z}+z_0|^{2p/d}\phi\big((K\tilde{z}+z_0)^{-p/d}K\cdot
    \tilde{x}\big)\Big), \]
  and the equation of $X_1$ is
  \[ K^{-d}(K\tilde{z}+z_0)^p + f(\tilde{x}) = 0. \]
  Similarly to the proof of
  Proposition~\ref{prop:decayest2} we now compare $K^{-2}\omega$ to 
  the product metric
  on $\mathbf{C}\times V_{K^{-d}z_0^p}$ with equation
  \[ K^{-d}z_0^p + f(\tilde{x}) = 0. \]
  The error given by the difference of the equations is of order
  $K^{1-d}z_0^{p-1} = O(K^{1-d}D^{p-1})$.

  Let us denote by $E$ the difference in K\"ahler potentials,
  \[ E = K^{-2}|K\tilde{z}+z_0|^{2p/d}\phi\big((K\tilde{z}+z_0)^{-p/d}K\cdot
    \tilde{x}\big) - K^{-2}|z_0|^{2p/d} \phi( Kz_0^{-p/d}\cdot
    \tilde{x}). \]
  Since
  \[ K(K\tilde{z}+z_0)^{-p/d} = z_0^{-p/d}K(1 + O(KD^{-1})), \]
  similarly to \eqref{eq:vv3}, \eqref{eq:vv4}, \eqref{eq:vv5} we get
  \[ |\nabla^i E| < C_i (|z_0|^{-p/d}K)^{-2-c}D^{-1}K =
    O(K^{-1-c}D^{\frac{(2+c)p}{d}-1}). \]
  Combining this with the error of order $K^{1-d}D^{p-1}$, we need
  $\delta > 2p/d$ such that
  \[ K^{1-d}D^{p-1} + K^{-1-c}D^{\frac{(2+c)p}{d}-1} < CD^\delta
    K^{-2}. \]
  This is equivalent to
  \[ K^{3-d}D^{p-1}\big(1 + (KD^{-p/d})^{d-2-c}\big) <
    CD^\delta. \]
  We can assume that $c>0$ is small, so that $d-2-c > 0$ (in fact from
  the estimate in Proposition~\ref{prop:hest1} we cannot expect that
  $c > d-2$ is possible). Then since $KD^{-p/d}$ is bounded away from
  zero, it is enough to choose $\delta$ so that
  \[ K^{3-d}D^{p-1}(KD^{-p/d})^{d-2-c} < CD^\delta, \]
  i.e.
  \[ K^{1-c}D^{\frac{(2+c)p}{d}-1} < CD^\delta. \]
  Since $K < D$, in order to be able to choose $\delta > 2p/d$, we need
  \[ 1-c + \frac{(2+c)p}{d}-1 > \frac{2p}{d}. \]
  This is equivalent to $p>d$, which holds in our setting.  
  \bigskip
  
 {\bf Region V}: Here $R < 2\kappa^{-1} \rho^{p/d}$, $\rho\in
 (D/2,2D)$, and $z$ is close to $z_0$ satisfying $|z_0|\sim D$.
 We rescale the metric by $|z_0|^{p/d}$. We introduce new coordinates 
  \[ \tilde{z} = z_0^{-p/d}(z-z_0), \quad \tilde{x} = z_0^{-p/d}\cdot x, \quad
    \tilde{r} = |z_0|^{-p/d}r, \]
   so that $|\tilde{z}| < 1$.  We have
  \[ |z_0|^{-2p/d}\omega = \ddb \Big[ |\tilde{z}|^2 +
    |z_0^{p/d}\tilde{z} + z_0|^{2p/d}|z_0|^{-2p/d} \phi\big(
    (z_0^{p/d}\tilde{z}+z_0)^{-p/d} z_0^{p/d} \cdot \tilde{x}\big)\Big], \]
  and the equation of $X_1$ is
  \[ z_0^{-p}(z_0^{p/d}\tilde{z}+z_0)^p + f(\tilde{x}) = 0. \]
  Note that 
  \[ z_0^{-p}(z_0^{p/d}\tilde{z}+z_0)^p = 1 + O(z_0^{p/d-1}) = 1 +
    O(D^{p/d-1}), \]
  and so comparing to $\mathbf{C}\times V_1$ with equation
  $1+f(\tilde{x})=0$ we introduce an error of order $D^{p/d-1}$.

  The difference in the corresponding K\"ahler potentials is
  \[ E = |z_0^{p/d}\tilde{z} + z_0|^{2p/d}|z_0|^{-2p/d} \phi\big(
    (z_0^{p/d}\tilde{z}+z_0)^{-p/d} z_0^{p/d} \cdot \tilde{x}\big) -
    \phi(\tilde{x}). \]
  Since
  \[ (z_0^{p/d}\tilde{z}+z_0)^{-p/d} z_0^{p/d} = 1 + O(z_0^{p/d-1}) =
    1 + O(D^{p/d-1}), \]
  we have $|\nabla^i E|_{|z_0|^{-2p/d}\omega} =O(D^{p/d - 1})$. 
  Therefore we need to be able to choose $\delta > 2p/d$ for which
  \[ D^{p/d-1} < CD^{\delta - 2p/d}. \]
  This is possible, since we assumed $p/d > 1$. 
\end{proof}

Abusing notation, let us now denote by 
$\omega$ a metric on $X_1\setminus \{0\}$ which agrees with the
$\omega$ constructed above on a neighborhood of $0$.  
The rest of the proof of Theorem~\ref{thm:2} is essentially the same
as the proof of Proposition~\ref{prop:perturbed}, in order to find a
$u$ such that
\[ \label{eq:e2} (\omega + \ddb u)^n = (\sqrt{-1})^{n^2} \Omega \wedge \overline{\Omega} \]
on the set where $\rho < A^{-1}$ for sufficiently large $A$. The key
ingredient is to find a right inverse for the Laplacian
\[ \label{eq:DD2}
  \Delta_\omega : C^{k,\alpha}_{\delta,\tau}(\rho^{-1}(0, A^{-1}]) \to
C^{k-2,\alpha}_{\delta-2, \tau-2}(\rho^{-1}(0, A^{-1}]), \]
for sufficiently large $A$,
where the weighted spaces are defined analogously to before. Here we
use smooth weight functions $\rho, w$ satisfying that $\rho^2$ agrees
with $|z|^2 + R^2$ near the origin (this is essentially the radial
distance from the singular point), and
\[ w = \begin{cases} 1\, &\text{ if } R > 2\kappa\rho, \\ 
  R / (\kappa \rho)\, &\text{ if } R\in (\kappa^{-1}\rho^{p/d},
  \kappa\rho), \\
  \kappa^{-2} \rho^{p/d-1}\, &\text{ if } R < \frac{1}{2}
  \kappa^{-1}\rho^{p/d}. \end{cases} 
\]
 The reason for the definition of $w$
is that near the singular rays $\mathbf{C}\times \{0\}$ our metric
$\omega$ is now modeled on $D^{2p/d}\omega_{V_1}$, transverse to the
$\mathbf{C}$-factor.

We have results analogous to
Propositions~\ref{prop:R<L} and \ref{prop:R>L}. Here we consider two
types of regions:
\[ \mathcal{U} = \{\rho < A^{-1}, R > \Lambda\rho^{p/d}\}, \]
for large $A, \Lambda$, and
\[ \mathcal{V} = \{ |z-z_0| < B|z_0|^{p/d}, \rho < A^{-1}, R <
  \Lambda\rho^{p/d}\}, \]
where $z_0, B$ are viewed as fixed. In an identical way to
Propositions~\ref{prop:R<L} and \ref{prop:R>L} we can define maps
$G:\mathcal{U}\to X_0$ and $H:\mathcal{V}\to \mathbf{C}\times V_1$. 
The Riemannian metrics $g, g_{X_0}, g_{\mathbf{C}\times V_1}$
defined by $\omega, \omega_0$ and the product metrics on
$\mathbf{C}\times V_1$ then satisfy
\[  \Vert G^*g_{X_0} - g\Vert_{C^{k,\alpha}_{0,0}} < \epsilon \]
if $\Lambda > \Lambda(\epsilon)$ and $A > A(\epsilon)$, and
\[ \Vert |z_0|^{2p/d} H^* g_{\mathbf{C}\times V_1} -
  g\Vert_{C^{k,\alpha}_{0,0}} < \epsilon, \]
for any $z_0, B, \Lambda$, once $A > A(\epsilon, \Lambda, B)$.

Just as in Proposition~\ref{prop:tancone}, these estimates imply that
the tangent cone of $(X_1,\omega)$ at 0 is given by the cone $X_0 =
\mathbf{C}\times V_0$. Moreover 
the existence of a right inverse for the Laplacian in \eqref{eq:DD2}
once $A$ is sufficiently large follows exactly the argument from
Section~\ref{sec:Laplaceinverse}. 

Proposition~\ref{prop:decayest3} implies that the Ricci potential $h$
of $\omega$ satisfies $h \in C^{0,\alpha}_{\delta'-2, \tau-2}$ for some
$\delta' > 2p/d$, and any $\tau < 0$. If follows that for slightly
smaller $\delta > 2p/d$ we have
\[ \Vert h\Vert_{C^{0,\alpha}_{\delta-2, \tau-2}(\rho^{-1}(0,A^{-1}])} <
CA^{\delta-\delta'}, \]
which can be made arbitrarily small by choosing $A$ large.  We can now
follow the proof of Proposition~\ref{prop:perturbed} to solve 
Equation~\ref{eq:e2} with $u\in \mathcal{B}$, where
\[ \mathcal{B} = \{ u\in C^{2,\alpha}_{\delta, \tau}\,:\, \Vert
  u\Vert_{C^{2,\alpha}_{\delta,\tau}} < \epsilon_0 \}, \]
for sufficiently small $\epsilon_0$, with $\tau < 0$ very close to 0. 
As in the proof of Proposition~\ref{prop:perturbed}, the key estimate
that we need is that
\[ \Vert u \Vert_{C^{2,\alpha}_{2,2}(\rho^{-1}(0, A^{-1}])} \leq C \Vert
u\Vert_{C^{2,\alpha}_{\delta, \tau}(\rho^{-1}(0, A^{-1}])}.  \]
This follows since the inequality  $w > C^{-1}
\rho^{p/d-1}$, together with $\delta > 2p/d$ implies
\[ \rho^\delta w^\tau < C\rho^2 w^2. \]
In particular if $u\in \mathcal{B}$ with sufficiently small
$\epsilon_0$, then $\omega + \ddb u$ is uniformly equivalent to
$\omega$. The rest of the argument is identical to the proof of
Proposition~\ref{prop:perturbed}. Finally, the tangent cone of $\omega + \ddb
u$ at 0 agrees with the tangent cone of $\omega$, since if $u\in
C^{2,\alpha}_{\delta, \tau}$ for our choices of $\delta, \tau$, then
$|\ddb u|_\omega\to 0 $ as $\rho\to 0$.

\end{document}